\documentclass[default,12pt]{sn-jnl}

\usepackage{graphicx}%
\usepackage{multirow}%
\usepackage{amsmath,amssymb,amsfonts}%
\usepackage{amsthm}%
\usepackage{mathrsfs}%
\usepackage[title]{appendix}%
\usepackage{xcolor}%
\usepackage{textcomp}%
\usepackage{manyfoot}%
\usepackage{booktabs}%
\usepackage{algorithm}%
\usepackage{algorithmicx}%
\usepackage{algpseudocode}%
\usepackage{listings}%
\usepackage{amsthm}
\usepackage{amssymb}
\usepackage{physics}
\usepackage[numbers]{natbib}
\usepackage{bm}
\usepackage{esint}
\usepackage{xcolor}
\numberwithin{equation}{section}

\usepackage{soul}

\theoremstyle{thmstyleone}%
\newtheorem{theorem}{Theorem}%
\newtheorem{proposition}[theorem]{Proposition}%

\newtheorem{lem}[theorem]{Lemma}

\theoremstyle{thmstyletwo}%
\newtheorem{remark}{Remark}%

\theoremstyle{thmstylethree}%
\newtheorem{definition}{Definition}%
\newtheorem{notat}{Notation} 
\numberwithin{theorem}{section}
\numberwithin{remark}{section}
\numberwithin{definition}{section}

\begin{document}

\title{On the effect of the Coriolis force on the enstrophy cascade}

\author*[1]{\fnm{Yuri} \sur{Cacchiò }}\email{yuri.cacchio@gssi.it}

\author*[2]{\fnm{Amirali} \sur{Hannani}}\email{amirali.hannani@kuleuven.be}

\author*[3]{\fnm{Gigliola} \sur{Staffilani}}\email{gigliola@math.mit.edu}

\affil[1]{\orgname{Gran Sasso Science Institute}, \orgaddress{\street{Viale Francesco Crispi, 7}, \city{L'Aquila}, \postcode{67100}, \state{Italy}}}

\affil[2]{\orgdiv{Instituut voor Theoretische Fysica}, \orgname{KU Leuven}, \orgaddress{\street{Celestijnenlaan 200d}, \city{Leuven}, \postcode{3001}, \state{Belgium}}}

\affil[3]{\orgdiv{Department of Mathematics}, \orgname{Massachusetts Institue of Technology}, \orgaddress{\street{77 Massachusetts Ave}, \city{Cambridge}, \postcode{02139-4307}, \state{MA}, \country{USA}}}

\abstract{We study the direct enstrophy cascade at small spatial scales in statistically stationary forced-dissipated 2D Navier-Stokes equations subject to the Coriolis force in the $\beta$-plane approximation.
We provide sufficient conditions inspired by \cite{bed2D,buler} to prove that at small scales, in the presence of the Coriolis force,  the so-called third-order structure function's asymptotics follows the third-order universal law of 2D turbulence without the Coriolis force. 
Our result indicates that at small scales, the enstrophy flux from larger to smaller
scales is not affected by the Coriolis force, confirming experimental and numerical observations.

To the best of our knowledge, this is the first mathematically rigorous study of the above equations. }
\maketitle
\tableofcontents
\section{Introduction}
Two-dimensional Navier-Stokes equations with additive stochastic noise is a canonical model for studying turbulence in two dimensions \cite{Batchelor2,Batchelor,eyink,Fjrtoft,frisch,Kraichnan,kuk,salmon}. 
Physically, considering these equations in two dimensions becomes relevant   
when we study fluids at large scales. For example, in geophysical motions, as a 
first approximation, we can consider two-dimensional flows due to the large 
aspect ratio (the ratio of lateral to vertical length scales) \cite{boffetta,Cope,lindborg,vallis1}.
 
Geophysical flows commonly appear on rotating planets \cite{gallagher1,holton,McWilliams,pedlosky,salmon,vallis1}. Taking into account  this 
motion, one should insert the effect of the Coriolis force into stochastic Navier-Stokes equations, resulting in
\begin{equation}\label{Problem}\left\{
    \begin{array}{rl}
    \partial_t u+\left(u\cdot \nabla\right)u+f u^\perp &=\nu\Delta u-\alpha u-\nabla p+\varphi;\\
    \nabla \cdot u&=0,
\end{array}
\right.
\end{equation}
where $u=(u^1,u^2)$ and $p$ are unknown velocity field and pressure, $(u\cdot\nabla)$ stands for the differential operator $u^1\partial_{x^1}+u^2\partial_{x^2}$ with $x=(x^1,x^2)$, $fu^\perp$ is the Coriolis force where $f$ is the Coriolis parameter (cf. \eqref{Coriolisapprox}) and $u^\perp=(-u^2,u^1)$, $\varphi$ is the stochastic process, and damping is provided by a combination of drag $\alpha>0$ and viscosity $\nu>0$. We refer the interested reader to \cite{salmon,vallis1} for complete derivations of the equations. 

In this paper, we use the well-known $\beta$-plane approximation
\begin{equation} \label{Coriolisapprox}
    f=f_0+\beta x^2
\end{equation}
where $f_0$ and $\beta$ are two constants that depend on a reference latitude. This regime is sufficient to capture the dynamic effects of
rotation (see \cite{salmon,vallis1} for further details).

As mentioned in our previous paper \cite{GigliolaAmiraliYuriWP} regarding the well-posedness of the equations \eqref{Problem}, introducing the Coriolis force in the $\beta$-plane approximation breaks the symmetry along the $x^2$ direction making the system anisotropic. Hence, we cannot consider a standard double periodic domain.
This leads us to take the most treatable, physically relevant domain: a periodic channel \cite[p 277]{salmon}. Then, we pose the system of equations \eqref{Problem} on a periodic domain in $x^1$, $\mathbb{T}_{x^1}=[0,L)$ torus of size $L$, and a bounded interval in $x^2$, $I=[a,b]$, equipped with periodic boundary conditions in $x^1$ and no-slip boundary conditions in $x^2$, i.e.
\begin{align}
    u(t,0,x^2)&=u(t,L,x^2),\label{periodicboundarycond}\\
    u(t,x^1,\partial I)&=0. \label{noslipboundary}
\end{align}
Finally, we assume that the stochastic process is spatially regular and white in time as in \cite{GigliolaAmiraliYuriWP}, i.e. 
\begin{equation}\label{stochasticprocess}
    \varphi(t,x)=\frac{\partial}{\partial t}\zeta(t,x),\ \ \ \ \zeta(t,x)=\sum_{j=1}^{\infty} b_j\beta_j(t)e_j(x),\ \ \ \ t\geq0
\end{equation}
where $\{e_j\}$ is a divergence-free orthonormal basis in $H$  (completion of divergence-free smooth functions with proper boundary conditions in $L^2(\mathbb{T}\times I)$,  cf. Definition \ref{spaces definitions}), $b_j$ are  constants such that, 
\begin{align} 
    \varepsilon&:=\frac{1}{2}\sum_{j=1}^\infty b_j^2<\infty, \label{eq: intro: energy}\\
    \eta&:=\frac{1}{2}\sum_j \fint_{\mathbb{T}\cross I}b_j^2|\nabla \cross e_j|^2<\infty, \label{eq: intro: enstrophyavg}
\end{align}
and $\{\beta_j\}$ is a sequence of independent standard Brownian motions (cf. \eqref{def:stochasticforce}). 

The constants $\varepsilon$, $\eta$ are called average energy and enstrophy input per unit time per unit area respectively. Notice that the energy input is independent of the viscosity $\nu$ and drag $\alpha$.

To get a complete picture of 2D turbulence it is crucial to introduce the notion of vorticity. In terms of vorticity $\omega$, \eqref{Problem} becomes
\begin{equation}\label{vorticityequation}
    \partial_t \omega+\left(u\cdot \nabla\right)\omega+\beta u^2 =\nu\Delta \omega-\alpha \omega+\nabla \cross \varphi
\end{equation}
where 
\begin{equation}\label{vorticitydef}
    \omega:=\nabla \cross u=\partial_1 u^2-\partial_2 u^1.
\end{equation}

The equations \eqref{Problem} with deterministic forcing (deterministic $\varphi$ or $\varphi=0$) have been studied both in mathematics and physics communities. In physics literature here we refer to \cite{Cost,Danilov,huang,nozawa,rhines,scott,ScottPolv,sriniv,vallis4,vallis3}, where this equation is used to model jet stream and turbulence. In the mathematics community we refer to \cite{mustafa,Yuri,Novak,pappalettera,gallagher2,gallagher1,Shepherd}  for questions such as well-posedness and long-time behavior. 

On the other hand, without the Coriolis force, the two-dimensional stochastic Navier-Stokes equations with additive noise (white in time and colored in space) have been studied extensively: we refer to \cite{Bensoussan,daprato,flandoli,Hairer,kuk,Temam}  concerning well-posedness, ergodicity, and long time behavior. As we mentioned, physically these equations are the canonical toy model to study turbulence in 2D \cite{Batchelor2,boffetta,frisch,kuk,Nazarenko,pedlosky,salmon,vallis1}. 

Having both the stochastic noise and the Coriolis force, the equations \eqref{Problem} are considered as a model of the so-called 
$\beta$-plane turbulence \cite{chekhlov,Cost2,Cost,Cope,galerpin2,galerpin1,huang2,salmon,vallis1}. 
In the mathematics community, we refer to \cite{GigliolaAmiraliYuriWP} concerning the well-posedness and associated energy estimates of \eqref{Problem}.

\subsection{Turbulence}
Turbulence in stationary forced dissipated fluids in \emph{3 dimensions} has  a long history. In his seminal work,  Kolmogorov  \cite{kolmogorov1,kolmogorov2,kolmogorov3} predicted the direct cascade of energy (cascade from small wave numbers to large wave numbers) under rather general assumptions. Quantitatively, the above cascade can be characterized using at least one of the following  well-known quantities: ensemble average  of the energy spectrum $k$
\begin{equation}
    E(k) \sim |k|\mathbb{E} (|\hat{u}(k)|^2),
\end{equation}
where $\hat{u}(k)$ is the Fourier transform of the velocity field, or  ensemble average of the $n$ point structure function defined as,
\begin{equation} \label{third order general}
    S_n(r):= \mathbb{E}(|\delta_r(u)|^n):= 
    \mathbb{E}(|u(x+r)-u(x)|^n),
\end{equation}
where $u$ is the velocity field, $r $ is the separating vector and the dependence on the spatial point $x$ is suppressed by assuming homogeneity. Assuming the flow to be isotropic, the celebrated Kolmogorov's $\frac{4}{5}$th law predicts the so called \emph{third order structure function} to behave as
\begin{equation}\label{thirdorderenergy}
    S_3^{\|}(r):= \mathbb{E}\left(\delta_r(u)\cdot\frac{r}{|r|}\right)^3 \sim -\frac45 \varepsilon |r|, \ \ \ l_\nu(\nu) \ll |r| \ll l_I,
\end{equation}
  where $l_\nu(\nu)$ is the scale where dissipation happens, $l_I$ is the scale where we inject energy, $\varepsilon$ denotes the injected energy, $S_3^{\|}$ is the third order \textit{longitudinal} structure function and $\delta_{r}(u):= u(x+r)-u(x)$.

  Kolmogorov's argument relies on the fact that we have only energy conservation. In \emph{2 dimensions}  phenomenology of statistically stationary, forced-dissipated flow changes drastically as we have two conservation laws:  energy and enstrophy. This leads to the more complicated \emph{dual cascade} picture, first argued by Fjrtoft in \cite{Fjrtoft}, and later by Batchelor and Kraichnan \cite{Batchelor2,Kraichnan}. In this regime, we have inverse energy cascade:  from small spatial scales (large wave numbers) to large spatial scales (small wave numbers), and direct enstrophy cascade from large spatial scales 
  to small spatial scales. 
  More explicitly, denote the dissipation, energy injection, and
 friction scale by $l_\nu$, $l_{I}$, and $l_\alpha$ respectively ($l_\nu(\nu) \ll l_{I} \ll l_\alpha(\alpha)$). The above dual cascade picture, predicts that energy is transferred from the injection scale to  the friction scale and damped by friction (\emph{inverse cascade}). On the other hand, enstrophy cascades from the injection scale to the dissipation scale and dissipates there by viscosity (\emph{direct cascade}).
 
In two dimensions, the analogous of the $\frac{4}{5}$th law \eqref{thirdorderenergy} has been realized only in 99' by Bernard, Lindborg, and Yakhot \cite{bernard,lindborg,yakhot}. The two-dimensional analogue is given in terms of the third-order structure functions assuming isotropy  as follows (recall 
$l_\nu \ll l_I \ll l_\alpha$)
\begin{subequations}
\begin{align}
S_3^{\|}(r)&=\mathbb{E}\left(\delta_r u\cdot\frac{r}{|r|}\right)^3\sim \frac{1}{8}\eta|r|^3, \ \ \
 l_\nu \ll |r| \ll l_I \label{thirdorederenstrophy2d},\\
S_3^{\|}(r)&=\mathbb{E}\left(\delta_r u\cdot\frac{r}{|r|}\right)^3\sim \frac{3}{2}\varepsilon|r|,  \ \ \ l_I \ll |r| \ll l_\alpha, \label{thirdorederenergy2d}
\end{align}
\end{subequations}
where as before $\varepsilon$ is the injected energy, and $\eta$ is the injected enstrophy.

We can also write \eqref{thirdorederenstrophy2d} in terms of the mixed velocity-vorticity structure function as predicted by Eyink \cite{eyink},
\begin{equation}\label{eq: velocity-vorticity str}
    S_3^{\|}(r)=\mathbb{E}\left(|\delta_r \omega|^2\delta_r u\cdot\frac{r}{|r|}\right)\sim -2\eta|r|, 
 \ \ \ l_\nu \ll |r| \ll l_I,
\end{equation}
originally derived in \cite{yaglom} for passive scalar
turbulence.

In $2D$, the energy spectrum is also predicted over the above-mentioned scales \cite{Batchelor2,Kraichnan}
\begin{subequations}
\begin{align}
    |k|&\mathbb{E}|\hat{u}(k)|^2\sim\varepsilon^{2/3}|k|^{-5/3},\ \ \ l^{-1}_{\alpha}\ll|k|\ll l^{-1}_I,\label{energyspectrum}\\
    |k|&\mathbb{E}|\hat{u}(k)|^2\sim\eta^{2/3}|k|^{-3},\ \ \ \ l^{-1}_{I}\ll|k|\ll l^{-1}_\nu.\label{enstrophyspectrum}
\end{align}
\end{subequations}

As we stated, the above picture has been rather well understood in the physics community by theoretical, and numerical arguments, and has been confirmed by experiments \cite{bernard,boffetta2,boffetta,eyink,huang2,Kraichnan2,lindborg,yakhot}. However, there are few rigorous mathematical works in this direction \cite{bed3D,bed2D,papathanasiou,buler}. Most notably we should mention \cite{bed3D,bed2D} where a very physically relevant and weak sufficient condition  (cf. Section \ref{1.4remark}) is provided for proving \eqref{thirdorederenstrophy2d}-\eqref{thirdorederenergy2d}, \eqref{eq: velocity-vorticity str} in 2D and \eqref{thirdorderenergy} in 3D for the counterpart of equations \eqref{Problem} \textbf{without the Coriolis force}.

\subsection{Result: Turbulence}
Adding the Coriolis force changes the behavior of solutions of \eqref{Problem}. Still we have two conservation laws: enstrophy and energy, which suggests that the dual cascade picture could persist. However, a priori it is not clear how the above laws \eqref{thirdorderenergy}, \eqref{thirdorederenstrophy2d}-\eqref{thirdorederenergy2d} and \eqref{eq: velocity-vorticity str} are modified in this situation.
This problem was first derived by Rhines in \cite{rhines} and argued later in \cite{chekhlov,Cost2,Cost,Cope,galerpin2,galerpin1,huang2,salmon,vallis1}. In these works, theoretical arguments suggest that at large scales the behavior of the spectrum will  change completely i.e. instead of \eqref{energyspectrum} we have
\begin{equation}\label{beta spectrum}
    |k|\mathbb{E}|\hat{u}| \sim \beta^2 |k|^{-5},
\end{equation}
where $\beta$ is the coefficient appearing in the beta plane approximation. On the other hand, numerical simulations and certain heuristics \cite{chekhlov,Cope,galerpin2,huang2,salmon,vallis1} suggest that at small scales we still have the same relations as  \eqref{thirdorederenstrophy2d}, \eqref{eq: velocity-vorticity str} and \eqref{enstrophyspectrum}.
The main goal of this paper is to provide a (rather weak) sufficient condition (see discussion on the assumption \eqref{uniform bound on vorticity} below) for proving the direct cascade of the enstrophy at small scales. 
In other words, we prove that at small scales, in the presence of the Coriolis force, \eqref{thirdorederenstrophy2d} and \eqref{eq: velocity-vorticity str} are valid assuming \eqref{uniform bound on vorticity}, meaning that at small scales the Coriolis force is negligible for the behavior of the third order structure function. 

\subsubsection{Statement of the Result}\label{statementofresult}
 More precisely, consider the 2D Navier-Stokes equations with additive white in time and colored in space noise with Coriolis force, in the $\beta$-plane approximation i.e. \eqref{Problem}. In \cite{GigliolaAmiraliYuriWP} we prove that the system \eqref{Problem} is well-posed and that these equations possess at least one stationary measure. By stationary we mean that the law of $u(\cdot)$ coincides with the law of $u(\cdot+ \tau)$ for any $\tau \geq 0$. Denote the expectation w.r.t this stationary solution by $\mathbb{E}$. Then we assume an uniform bound of $\mathbb{E}\norm{\omega}^2_{L^2}$ in $\alpha$ and $\nu$ (cf. \eqref{uniform bound on vorticity}), we refer to Section \ref{1.4remark} for the motivation related to this assumption. Our main result stated in Theorem \ref{thm:enstrophycascade}, illustrates the following relations concerning the third order structure function at small scales ($l_{\nu} \ll l \ll l_{I}$),  
\begin{align}
    \mathbb{E}&\fint_\mathbb{S}\fint_{\mathbb{T}\cross I}\left|\delta_{ln}\omega\right|^2\delta_{ln}u \cdot n dxdn \sim -2\eta l, \label{into:directcascade3.1a}\\
    \mathbb{E}&\fint_\mathbb{S}\fint_{\mathbb{T}\times I}|\delta_{ln}u|^2\delta_{ln}u\cdot n\ dxdn\sim \frac{1}{4}\eta l^3,\label{intro:directcascade3.1b}
\end{align}
where $\eta$ is the injected enstrophy by the stochastic force \eqref{eq: intro: enstrophyavg}, $\delta_{ln}u= u(x+nl)-u(x)$, $n$ denotes the normal vector and $l$ is a scalar, $\fint$ is the normalized (w.r.t domain size) integration. 

As expressed, the above asymptotics should be understood for small scales $l$ with 
$l_{\nu} \ll l \ll l_{I} \ll l_{\alpha}$, in the inviscid limit ($\nu \to 0$). To make this more precise, first, we fix a scale $l \in [l_{\nu},l_{I}] $. In order to realize the limit $l \to 0$, we  take the limit 
$\nu, \alpha \to 0$, notice that in this limit $l_{\nu}(\nu) \to 0$, and $l_{\alpha}(\alpha) \to \infty$. Only afterward, we take the limit $l_{I} \to 0$ (cf. Theorem \ref{thm:enstrophycascade}).   
On the other hand, in an ongoing project an inverse cascade of energy for statistically stationary solution of \eqref{Problem} at large scales is considered. Here we analyze the fact that  the third order structure function behavior is only determined by the Coriolis force. More precisely, in an appropriate inertial range, $l_{\nu}  \ll l_{I}\ll l \ll l_{\alpha}$, we expect 
\begin{equation}
\mathbb{E}\fint_\mathbb{S}\fint_{\mathbb{T}\times I}|\delta_{ln}u|^2\delta_{ln}u\cdot n\ dxdn\sim C \varepsilon\left(\frac{\beta}{\alpha}\right)^n r^m \text{ with }n\geq1\text{ and } m>1 
\end{equation}
as suggested by its spectral counterpart \eqref{beta spectrum}.

\subsection{Sketch of the proof and main ideas}\label{sec: proof idea}

Let us briefly sketch the proof's ideas and highlight the new difficulties arising in our setup. 

To prove the cascade, we follow the ideas sketched in \cite{bed3D,bed2D,buler}. We follow the evolution of the two-point correlation function, and using stationarity we obtain the so called KHM (Karman-Howarth-Monin) relations \cite{karman,monin}. Let us emphasize that because of the Coriolis force, we have a new term in our modified KHM (cf. Section \ref{sectionKHM}). The process of dealing with each terms, after taking the limits, in the KHM relation is morally similar to \cite{bed2D} at small scales.
Moreover, for treating these terms certain regularity properties are needed. These properties are straightforward in the case of \cite{bed2D} owing to the existing results in the literature, whereas in our case we cannot appeal to these results due to the presence of the new $\beta$ term. Therefore we proved similar (in fact weaker but still sufficiently strong) regularity of these terms in \cite{GigliolaAmiraliYuriWP}. 
More importantly, in our case we need to deal with a new term corresponding to the Coriolis force in the KHM. This constitutes the main part of our analysis. By using the regularity results concerning the invariant measure's support, we are able to  Taylor expand these terms and prove that low-order contributions are miraculously zero, and higher orders are dominated by the effect of $\eta$ at small scales.

\subsection{Remarks}\label{1.4remark}
In this section we make some remarks concerning our result. First, let us explain the physical relevance of the uniform bound on $\mathbb{E}\norm{\omega}^2_{L^2}$ (cf. \eqref{uniform bound on vorticity}).

By using the It\^{o} formula (cf. \cite[Theorem 7.7.5]{kuk}) we obtain in \cite{GigliolaAmiraliYuriWP} the following energy and enstrophy balance for stationary solution: 
stationary energy balance 
\begin{equation}\label{energy balance}
    \alpha\mathbb{E}\norm{u}^2_{L^2}+ \nu\mathbb{E}\norm{\nabla u}^2_{L^2}=\varepsilon,
\end{equation}
stationary enstrophy balance 
\begin{equation}\label{enstrophy balance}
\alpha\mathbb{E}\norm{\omega}_{L^2}^2 +\nu\mathbb{E} \norm{\nabla \omega}_{L^2}^2 =\eta,
\end{equation}
where $u$ is a statistically stationary solution and $\mathbb{E}$ is the stationary measure associated with $u$.
Taking advantage of the above balance relations we may derive from \eqref{uniform bound on vorticity} the following limits: 
\begin{align}
     &\lim_{\nu \to 0} \sup_{\alpha\in(0,1)}\left|\varepsilon-\alpha\mathbb{E}\norm{u}^2_{L^2}\right|=0,\label{eq:WAD2}\\
    &\lim_{\alpha \to 0}\sup_{\nu\in(0,1)}\left|\nu\mathbb{E}\norm{\nabla\omega}^2_{L^2}-\eta\right|=0.\label{eq:WAD4}
\end{align}
Physically, the above expressions mean that in the limit $\alpha,\nu \to 0$, all the injected energy by the stochastic force will be dissipated thanks to the drag $\alpha$ at large scales. Moreover, all the enstrophy injected by this force will be dissipated at small scales by viscosity. Notice that this is aligned with the "natural"  anomalous dissipation assumption appearing in theory of turbulence \cite{frisch}. This assumption states that Navier Stokes equations can balance the external energy independent of the Reynolds number thanks to the effect of the non-linearity. 

As we mentioned, our work is inspired by \cite{bed3D,bed2D} and also \cite{buler}. Our main contribution is to consider the effect of the planetary rotation (Coriolis force). We already explained the physical application of these equations and the difference with the case studied above. Mathematically, dealing with the new term poses several challenges as mentioned in Section \ref{sec: proof idea}.

Let us mention that the averaging over the circle in order to deal with anisotropy is also discussed in \cite{arad,bed3D,kurien,nie}. However, there is a crucial difference in our case: our equations are intrinsically anisotropic because of the Coriolis force, and this average is essential to obtain a quantity only depending on the scale $l$.  

From a physical point of view, it is not obvious that all the contribution of the Coriolis force on the dual cascade vanishes. In fact, in the $\beta$-plane approximation $f_0$ represents the horizontal component of the rotation while $\beta$ represents the variation of the rotation in relation to the curvature of the planets \cite{vallis1}. One might predict that at small scales the effect of curvature $\beta$ is zero but we can not make a priori estimates regarding the other component $f_0$.
In contrast, it can be expected that at large scales the $\beta$ contribution becomes relevant but still, we can not make a priori estimates on $f_0$.

However, for 2d incompressible flow, $f_0$ has no impact on the dynamical evolution at all, then in Section \ref{sectionKHM} the new term derived in the KHM due to the Coriolis force, which we call $\varTheta$, depends only on $\beta$. Moreover, at small scales we show that the contribution of $\varTheta$ on the direct enstrophy cascade is zero, confirming the above-mentioned heuristics as well as numerical observations \cite{chekhlov,Cost2,Cost,Cope,galerpin2,galerpin1,huang2}.

The organization of the paper is as follows. In Section \ref{sectionprelimin}, we state the assumptions, certain terminology and the main results. We then derive in Section \ref{sectionKHM} the modified velocity/vorticity KHM relations related to the Coriolis force. In Section \ref{sectiondirect} we prove the main theorem which shows the direct cascade of enstrophy.

\section{Preliminaries and Main Results}\label{sectionprelimin}
Let us introduce the following vector spaces and set some notations.
\begin{definition}\label{spaces definitions}
We define 
\begin{itemize}
    \item[1)] $D
    =\{u\in C^\infty(\mathbb{T}\cross (a,b)); u(0,x^2)=u(L,x^2) \text{ with compact support in }x^2\}$
    \item[2)] $D_\sigma
    =\{u\in D(\mathbb{T}\cross (a,b)), \nabla \cdot u=0\}$
    \item[3)] $H^m_0
    =\overline{D(\mathbb{T}\cross (a,b))}^{H^m(\mathbb{T}\cross (a,b))}$
    \item[4)] $H
    =\overline{D_\sigma(\mathbb{T}\cross (a,b))}^{L^2(\mathbb{T}\cross (a,b))}$
    \item[5)] $V
    =\overline{D_\sigma(\mathbb{T}\cross (a,b))}^{H^1_0(\mathbb{T}\cross (a,b))}$
\end{itemize}
where $\overline{D(\mathbb{T}\cross (a,b))}^{H^m(\mathbb{T}\cross (a,b))}$ denotes the closure of $D$ in the Sobolev space $H^m$. A similar notation holds for $H$ and $V$.
In addition, $H$ and $V$ are equipped with the induced norm,
\begin{equation*}
    \norm{\cdot}_{L^2(\mathbb{T}\cross I)}; \ \ \ \norm{\cdot}_{H^1(\mathbb{T}\cross I)},
\end{equation*}
respectively. $H'$ and $V'$ denote the dual space (w.r.t the usual inner product $<\cdot,\cdot>$)  of $H$ and $V$.
\end{definition}
Assuming boundary conditions as in \eqref{periodicboundarycond}, \eqref{noslipboundary} and fixing the initial data 
\begin{equation}
    u(0)=u_0\in H,
\end{equation}
we derive the following Cauchy problem 
\begin{equation}\label{Problemleray}\left\{
    \begin{array}{rl}
     \partial_t u+\nu Au+B(u)+\alpha u+Fu&=\varphi; \\
     \nabla \cdot u&=0;\\
     u(t,0,x^2)&=u(t,L,x^2);\\
     u(t,x^1,\partial I)&=0;\\
     u(0)&=u_0\in H 
\end{array}
\right.
\end{equation}
where 
\begin{align}
    A u&=-\Pi\Delta u, \\
    B(u)&=B(u,u)=\Pi((u\cdot\nabla)u),\\
    Fu&=\Pi(fu^\perp),
\end{align}
and
\begin{equation}
    \Pi:H^s(Q,\mathbb{R}^2)\to H^s_\sigma(Q,\mathbb{R}^2), 
\end{equation}
is the orthogonal projection for any bounded Lipschitz domain $Q$ called Leray projection. Here $H^s$ is the standard Sobolev space and $H^s_\sigma$ denotes the $H^s$-divergent free functions. 
The system \eqref{Problemleray} is a weak formulation of \eqref{Problem} in the sense of Leray formulation of Navier-Stokes equations \cite{Leray1,Leray3,Leray2}.  We refer to \cite{GigliolaAmiraliYuriWP} for the equivalence between \eqref{Problem} and\eqref{Problemleray}.

We assume that 
\begin{equation} \label{def:stochasticforce}
    \varphi(t,x)=\frac{\partial}{\partial t}\zeta(t,x),\ \ \ \ \zeta(t,x)=\sum_{j=1}^{\infty} b_j\beta_j(t)e_j(x),\ \ \ \ t\geq0
\end{equation}
where $\{e_j\}$ is a divergence-free orthonormal basis in $H$, $b_j$ are  constants satisfying \eqref{eq: intro: energy}-\eqref{eq: intro: enstrophyavg}, and $\{\beta_j\}$ is a sequence of independent standard Brownian motions. 
Moreover, the Brownian motions $\{\beta_j\}$ are defined on a complete probability space $(\Omega,\mathcal{F},\mathbb{P})$ with a filtration $\mathcal{G}_t$, $t\geq0$, and the $\sigma-$algebras $\mathcal{G}_t$ are completed with respect to $(\mathcal{F},\mathbb{P})$, that is, $\mathcal{G}_t$ contains all $\mathbb{P}-$null sets $A\in \mathcal{F}$.

Under these assumptions, we have existence and uniqueness of the solution of \eqref{Problemleray} in terms of the following definition (see \cite{GigliolaAmiraliYuriWP} for the details).
\begin{definition}\label{solutiondef}
An $H-$valued random process $u(t)$, $t\geq0$, is called a solution for \eqref{Problemleray} if:
\begin{itemize}
    \item [1)] The process $u(t)$ is adapted to the filtration $\mathcal{G}_t$ (cf. \eqref{def:stochasticforce}), and its almost every trajectory belongs to the space
    \begin{equation*}
        \mathcal{X}=C\left(\mathbb{R}_+; H\right)\cap L^2_{loc}\left( \mathbb{R}_+;V\right). 
    \end{equation*}
    \item[2)] Identity \eqref{Problemleray} holds in the sense that, with probability one,
    \begin{equation}\label{solution}
        u(t)+\int_0^t 
        \left(\nu Au+B(u)+\alpha u+Fu\right)ds=u(0)+\zeta(t), \ \ \ \ t\geq0,
    \end{equation}
    where the equality holds in the space $H^{-1}$.
\end{itemize}
\end{definition}
\begin{remark}
   1) In \cite{GigliolaAmiraliYuriWP} we proved that \eqref{Problem} admits a solution in terms of Definition \eqref{solutiondef}. Then, we can directly apply \cite[Corollary 2.4.11.]{kuk} concerning the regularity of the solution. In particular, the Corollary states that both for periodic and Dirichlet boundary conditions we have the following estimates,
\begin{equation}
    \mathbb{E}\norm{u}^r_{L_t^\infty L^2_x}< \infty, \ \ \ \mathbb{E}\norm{\nabla u}^2_{L^2_{t,x}}< \infty, \ \ \ \forall r\in[1,\infty).
\end{equation} 
2) As stated in \cite[Corollary 2.2]{GigliolaAmiraliYuriWP}, we have a similar well-posedness result for \eqref{vorticityequation}.
\end{remark}
From here on we assume that $u$ is a statistically stationary solution of \eqref{Problemleray} and we denote by $\mathbb{E}$ the stationary measure associated with $u$.

We study system \eqref{Problem} by assuming an uniform bound in $\alpha,\nu$ on the expectation of the vorticity as follows 
\begin{equation}\label{uniform bound on vorticity}
    \mathbb{E}\norm{\omega}^2_{L^2}\leq c<\infty.
\end{equation}
This means that the enstrophy associated with the vorticity does not blow up as dissipation vanishes. 
The above uniform bounds lead to stronger regularity properties of solutions, necessary for proving the convergence results in the next Theorem. In fact, we need to recover smoothness since the presence of the Coriolis force in the $\beta$-plane approximation reduces regularity making the system strongly anisotropic. 

We note that \eqref{uniform bound on vorticity} implies the more general assumption called weak anomalous dissipation (WAD) regime \cite[Sec. 1.2]{bed3D}, \cite[Sec. 1.2]{bed2D} which is inspired by \cite{eyink,Kupiainen}.

Under these assumptions, we prove the following theorem,
\begin{theorem}\label{thm:enstrophycascade}
Suppose that $L=L(\alpha)<\infty$ and $b=b(\alpha)<\infty$ are a continuous monotone decreasing functions such that $\lim_{\alpha\to 0}L=\lim_{\alpha\to 0}b=\infty$ and suppose that $a=a(\alpha)>-\infty$ is a continuous monotone increasing function such that $\lim_{\alpha\to 0}a=-\infty$. Let $\{u\}_{\nu,\alpha>0}$ be a sequence of statistically stationary solutions of \eqref{Problemleray} such that \eqref{uniform bound on vorticity} holds. Then there exists $l_\nu\in(0,1)$ satisfying $\lim_{\nu\to0}l_{\nu}=0$ such that 
\begin{align}
    &\lim_{l_I\to0}\limsup_{\nu,\alpha\to 0}\sup_{l\in[l_\nu,l_I]}\left|\frac{1}{l}\mathbb{E}\fint_\mathbb{S}\fint_{\mathbb{T}\cross I}\left|\delta_{ln}\omega\right|^2\delta_{ln}u \cdot n dxdn+2\eta\right|=0\label{directcascade3.1a}\\
    &\lim_{l_I\to 0}\limsup_{\nu,\alpha\to0}\sup_{l\in[l_{\nu},l_I]}\left|\frac{1}{l^3}\mathbb{E}\fint_\mathbb{S}\fint_{\mathbb{T}\times I}|\delta_{ln}u|^2\delta_{ln}u\cdot n\ dxdn-\frac{1}{4}\eta\right|=0\label{directcascade3.1b}
\end{align}
In particular, it suffices to choose $l_{\nu}\to0$ satisfying 
\begin{equation}\label{conditionlv}
   \nu^{1/2}=o_{\nu\to 0}(l_\nu).
\end{equation}
\end{theorem}


\begin{remark}
    We can consider in \eqref{Problemleray} a more generalized linear Ekman-type damping $-\alpha(-\Delta)^{2j}u$ with $j\geq0$ instead of $-\alpha u$ as in \cite{bed2D}. However, this does not influence the results. Then we assume $j=0$ for notation simplicity.
    A similar argument holds if we replace viscosity $\nu\Delta$ with hyperviscosity $-\nu(-\Delta)^j$ with $j>1$.
\end{remark}

\section{Modified Velocity KHM relations}\label{sectionKHM}

In this section, we derive the Modified Velocity KHM relations initially in terms of spherically averaged coordinates. The new term is denoted by $\overline{\varTheta}$.
Since we are not in a periodic space along $x^2$, it is first necessary to extend the velocity $u$ to the space $\mathbb{T}\cross \mathbb{R}$. We note that this extension preserves the regularity properties of the solution $u$ (see Remark \ref{reg ext}).

\subsection{Solution extension}

Let $u$ be a statistically stationary solution of \eqref{Problemleray} in the sense of Definition \ref{solutiondef}. Since $u\in V=\overline{D_\sigma(\mathbb{T}\cross (a,b))}^{H^1_0(\mathbb{T}\cross (a,b))}$ there exist a sequence of smooth function $\{u_n\}_n\in D_\sigma$ such that $u_n\to u$ for $n\to \infty$. Then, we define for each $n\in \mathbb{N}$,
\begin{equation}
    \widetilde{u}_n(x)=\left\{
    \begin{array}{rl}
    u_n, &\text{ if } x^2\in I;\\
    0, &\text{ otherwise.}
\end{array}
\right.
\end{equation}
Hence, $\{\widetilde{u}\}_n$ is a sequence of smooth compactly supported functions that converges in $H^1$. Its limit, that we denote by $\widetilde{u}\in H^1_0(\mathbb{T}\cross\mathbb{R})$ is the extension of $u$.

Similarly, we define the extensions $\widetilde{p}$ and $\widetilde{e_j}$.

\begin{remark}(Regularity of the extension)\label{reg ext}
\begin{itemize}
    \item  Since $\div u=0$, then $\div\widetilde{u}=0$. In fact, for any test function $\phi\in C^\infty_c$ we have 
\begin{equation}
    \langle\div\widetilde{u},\phi\rangle_{L^2(\mathbb{T}\cross\mathbb{R})}=-\langle\widetilde{u},\grad\phi\rangle_{L^2(\mathbb{T}\cross\mathbb{R})}=-\langle u,\grad\phi\rangle_{L^2(\mathbb{T}\cross I)}=\langle \div u,\phi\rangle_{L^2(\mathbb{T}\cross I)}=0.
\end{equation}
  \item $\widetilde{u}$ is a solution in the sense of Definition \ref{solutiondef} of the following problem,
\begin{equation}\label{Problemleraytilda}\left\{
    \begin{array}{rl}
     \partial_t \widetilde{u}+\nu A\widetilde{u}+B(\widetilde{u})+\alpha \widetilde{u}+F\widetilde{u}&=\widetilde{\varphi}; \\
     \nabla \cdot \widetilde{u}&=0;\\
     \widetilde{u}(t,0,x^2)&=\widetilde{u}(t,L,x^2);\\
     \widetilde{u}(0)&=\widetilde{u}_0\in H.
\end{array}
\right.
\end{equation}
In fact since $\textit{supp}(\widetilde{u})\subset \mathbb{T}\cross I$, for any $\phi\in C^\infty_c(\mathbb{T\cross R})$ we have
\begin{align*}
    &\langle \widetilde{u},\phi \rangle_{\mathbb{T\cross R}}+\int_0^t  \langle \nu A\widetilde{u}+B(\widetilde{u})+\alpha\widetilde{u}+F\widetilde{u},\phi \rangle_{\mathbb{T\cross R}}ds- \langle \widetilde{u(0)}+\widetilde{\zeta}(t),\phi \rangle_{\mathbb{T\cross R}}\\
    =&\langle u,\phi \rangle_{\mathbb{T\cross } I}+\int_0^t  \langle \alpha u+Fu,\phi \rangle_{\mathbb{T\cross }I}ds+\int_0^t  \nu\langle  u,\Delta\Pi\phi \rangle_{\mathbb{T\cross }I}ds\\
    &-\int_0^t  \langle  u\otimes u,\nabla\Pi\phi \rangle_{\mathbb{T\cross }I}ds- \langle u(0)+\zeta(t),\phi \rangle_{\mathbb{T\cross }I}=0
\end{align*}
where we use $\div(\widetilde{
u}\otimes\widetilde{u})=(\widetilde{u}\cdot\nabla)\widetilde{u}$, the orthogonal property of the Leray projection and \eqref{solution}. As mentioned before, solving \eqref{Problemleraytilda} is equivalent to solve
\begin{equation}\label{Problem1}
\left\{\begin{array}{rl}
    \partial_t \widetilde{u}+\div(\widetilde{u}\otimes\widetilde{u})+f \widetilde{u}^\perp &=\nu\Delta \widetilde{u}-\alpha \widetilde{u}-\nabla\widetilde{p}+\widetilde{\varphi}; \\
     \nabla \cdot \widetilde{u}&=0;\\
     \widetilde{u}(t,0,x^2)&=\widetilde{u}(t,L,x^2);\\
     \widetilde{u}(0)&=\widetilde{u}_0\in H.
\end{array}
\right.
\end{equation}
\end{itemize}

\end{remark}

\subsection{Weak Velocity KHM relation}
We define the extended two-point correlation for the velocity, Coriolis force, pressure and noise as
\begin{align}
    \Gamma(y)&:=\mathbb{E}\fint_{\mathbb{T}\times \mathbb{R}}\widetilde{u}(x)\otimes \widetilde{u}(x+y)dx,\\
    \varTheta(y)&:=\frac{1}{2}\mathbb{E}\fint_{\mathbb{T}\times \mathbb{R}}\widetilde{u}(x)\otimes (f\widetilde{u^\perp})(x+y)+(f\widetilde{u^\perp})(x)\otimes \widetilde{u}(x+y)dx,\\
     P(y)&:=\mathbb{E}\left(\nabla_y\int_{\mathbb{T}\cross \mathbb{R}} \widetilde p(x)\widetilde u(x+y)dx-\nabla^\top_y\int_{\mathbb{T}\cross \mathbb{R}} \widetilde u(x)\widetilde p(x+y)dx\right)\\
    a(y)&:=\frac{1}{2}\sum_{j}b_j^2\fint_{\mathbb{T}\times \mathbb{R}}\widetilde{e_j}(x)\otimes \widetilde{e_j}(x+y)dx,
\end{align}
for any $y\in \mathbb{R}^2$, where $f(x)=f_0+\beta x^2$, and the corresponding regularized flux structure function, "third order structure-function", defined for each $j=1,2$ by
\begin{equation}\label{dj definition}
    D^j(y)=\mathbb{E}\fint_{\mathbb{T}\times \mathbb{R}} (\delta_y \widetilde{u}(x)\otimes\delta_y \widetilde{u}(x))\delta_y \widetilde{u}^j(x) dx,
\end{equation}
where $\delta_y \widetilde{u}(x):=\widetilde{u}(x+y)-\widetilde{u}(x)$ and with an abuse of notation $\fint=\frac{1}{|\mathbb{T}\times I|}\int$

\begin{remark}
We cannot directly define 
\begin{equation}
    \mathbb{E}\fint_{\mathbb{T}\times I}u(x)\otimes u(x+y)dx
\end{equation}
because $(x+y)\not\in \mathbb{T}\cross I$ for any $y\in\mathbb{R}^2$. Due to the solution extension constructed in the previous section, $\Gamma(y)$ (and similarly for the other quantities) is well-defined. Moreover,
    \begin{equation}
\mathbb{E}\fint_{\mathbb{T}\times \mathbb{R}}\widetilde{u}(x)\otimes \widetilde{u}(x+y)dx=\mathbb{E}\fint_{\mathbb{T}\times I}\widetilde{u}(x)\otimes \widetilde{u}(x+y)dx,
    \end{equation}
since $supp(\widetilde{u})\subset I$.
\end{remark}

\begin{remark}\label{rmk1}
We observe that the quantities defined above are uniformly bounded and at least $C^3$. This follows from a regularization argument and an application of Ascoli-Arzelà Theorem.
\end{remark}

\begin{notat}
    Given any two rank tensors A and B we will denote the Frobenius product by $A:B=\sum_{i,j}A_{ij}B_{ij}$ and the norm $|A|=\sqrt{A:A}$.
\end{notat}
\begin{proposition}\label{weakvelocitykhmprop} (Weak Velocity KHM relation). Let $\widetilde{u}$ be a statistically stationary solution to \eqref{Problem1} and let $\phi(y)=(\phi_{ij}(y))^2_{ij=1}$ be a smooth compactly supported test function of the form
\begin{equation}\label{testfunction}
    \phi(y)=\Phi(|y|)Id+\Psi(|y|)\hat{y}\otimes\hat{y}, \ \ \ \hat{y}=\frac{y}{|y|},
\end{equation}
where $\Phi$ and $\Psi$ are smooth and compactly supported on $(0,\infty)$.
Then, the following identity holds
\begin{align}\label{KHM}
\sum^2_{j=1}\int_{\mathbb{R}^2}\partial_j\phi(y):D^j(y)dy&=4\nu\int_{\mathbb{R}^2}\Delta\phi(y):\Gamma(y)dy-4\alpha\int_{\mathbb{R}^2}\phi(y):\Gamma(y)dy\notag\\
&-4\int_{\mathbb{R}^2}\phi(y):\varTheta(y)dy
+4\int_{\mathbb{R}^2}\phi(y):a(y)dy.
\end{align}
\end{proposition}

\begin{proof}
Let us consider the standard mollifier
and convolving $\widetilde{u}$ with it, i.e.,
\begin{equation}\label{moll u}
    \widetilde{u}^\varepsilon(x)=\rho_\varepsilon* \widetilde{u}=\int \varepsilon^{-2}\rho\left(\frac{y}{\varepsilon}\right)\widetilde{u}(x-y)dy.
\end{equation}
Then, the mollifier stochastic equation for $\widetilde{u}^\varepsilon$ as
\begin{align}\label{stochasticeq}
    d\widetilde{u}^\varepsilon(t,x)+\nabla \cdot (\widetilde{u} \otimes \widetilde{u})^\varepsilon(t,x)dt-\nu \Delta \widetilde{u}^\varepsilon(t,x) dt
    +\alpha \widetilde{u}^\varepsilon(t,x) dt&+ (f\widetilde{u^\perp})^\varepsilon (t,x)dt\\
    &+\nabla \widetilde{p}^\varepsilon(t,x)dt=d\widetilde{\zeta^\varepsilon}(t,x)\notag
\end{align}
where we used the identity 
\begin{equation*}
    \nabla \cdot (\widetilde{u}\otimes \widetilde{u})=(\nabla \cdot \widetilde{u})\widetilde{u}+(\widetilde{u}\cdot \nabla)\widetilde{u}=(\widetilde{u}\cdot \nabla)\widetilde{u}.
\end{equation*}
We can see equation \eqref{stochasticeq} as a finite-dimensional SDE.
Using the abstract Ito's formula in Hilbert spaces (cf. \cite[Sec. 7.7]{kuk}) one can obtain the following stochastic products rule (cf. \cite[Proposition 3.1]{bed3D} for similar computation).  
 Let $y\in\mathbb{T}\times \mathbb{R}$, then the evolution equation of $\widetilde{u}^\varepsilon(t,x)\otimes \widetilde{u}^\varepsilon(t,x+y)$ satisfies the stochastic product rule 
\begin{align}\label{productrule}
    d\left(\widetilde{u}^\varepsilon(t,x)\otimes \widetilde{u}^\varepsilon(t,x+y)\right)&=d\widetilde{u}^\varepsilon(t,x)\otimes \widetilde{u}^\varepsilon(t,x+y)+\widetilde{u}^\varepsilon(t,x)\otimes d\widetilde{u}^\varepsilon(t,x+y)\notag\\
    &+\left[\widetilde{u}^\varepsilon(\cdot,x), \widetilde{u}^\varepsilon(\cdot,x+y)\right](t),
\end{align}
where $\left[\widetilde{u}^\varepsilon(\cdot,x), \widetilde{u}^\varepsilon(\cdot,x+y)\right](t)$ is the cross variation of $\widetilde{u}^\varepsilon(\cdot,x)$ and $\widetilde{u}^\varepsilon(\cdot,x+y)$ and is given by
\begin{align*}
    \left[\widetilde{u}^\varepsilon(\cdot,x), \widetilde{u}^\varepsilon(\cdot,x+y)\right](t)&=\left[\widetilde{\zeta}^\varepsilon(\cdot,x),\widetilde{\zeta}^\varepsilon(\cdot,x+y)\right](t)\\
    &=\sum_{i,j}b_i b_j\left(\widetilde{e_i}^\varepsilon\otimes \widetilde{e_j}^\varepsilon(x+y)\right)\left[\beta_i,\beta_j\right](t)\\
    &=tC^\varepsilon(x,x+y),
\end{align*}
with,
\begin{equation*}
    C^\varepsilon(x,x+y)=\sum_{j=1}^\infty b_j^2 \widetilde{e_j}^\varepsilon(x)\otimes \widetilde{e^\varepsilon}_j(x+y). 
\end{equation*}
We observe that 
\begin{equation}
    \nabla \cdot (\widetilde{u} \otimes \widetilde{u})^\varepsilon(x,t)\otimes\widetilde{u}^\varepsilon(t,x+y)=\sum_{j=1}^2 \partial_{x^j}(\widetilde{u}^j \widetilde{u})^\varepsilon(t,x)\otimes \widetilde{u}^\varepsilon(t,x+y)
\end{equation}
Then, integrating \eqref{productrule} in $x$, using \eqref{stochasticeq} and due some integration by parts we obtain the following SDE
\begin{align*}
    &d\bigg(\fint_{\mathbb{T}\times \mathbb{R}}\widetilde{u}^\varepsilon(t,x)\otimes \widetilde{u}^\varepsilon(t,x+y)dx\bigg)\\
    &=\sum_{j=1}^2\left(\partial_{y^j}\fint_{\mathbb{T}\cross \mathbb{R}}(\widetilde{u}^j \widetilde{u})^\varepsilon(t,x)\otimes \widetilde{u}^\varepsilon (t,x+y)-\widetilde{u}^\varepsilon(t,x)\otimes(\widetilde{u}^j \widetilde{u})^\varepsilon(t,x+y)dx\right)dt\\
    &+2\nu\left(\Delta_y\fint_{\mathbb{T}\times \mathbb{R}}\widetilde{u}^\varepsilon(t,x)\otimes \widetilde{u}^\varepsilon(t,x+y)dx\right)dt-2\alpha\left(\fint_{\mathbb{T}\times \mathbb{R}}\widetilde{u}^\varepsilon(t,x)\otimes \widetilde{u}^\varepsilon(t,x+y)dx\right)dt\\
    &-\left(\fint_{\mathbb{T}\times \mathbb{R}}(f\widetilde{u^\perp})^\varepsilon(t,x)\otimes \widetilde{u}^\varepsilon(t,x+y)+\widetilde{u}^\varepsilon(t,x)\otimes (f\widetilde{u^\perp})^\varepsilon(t,x+y)dx\right)dt\\
    &+\left(\fint_{\mathbb{T}\times \mathbb{R}}C^\varepsilon(x,x+y)dx\right)dt\\
    &+\left(\nabla_y\int_{\mathbb{T}\cross \mathbb{R}} \widetilde{p}^\varepsilon(t,x)\widetilde{u}^\varepsilon(t,x+y)dx-\nabla^\top_y\int_{\mathbb{T}\cross \mathbb{R}} \widetilde{u}^\varepsilon(t,x)\widetilde{p}^\varepsilon(t,x+y)dx\right)dt\\
    &+\sum_{j=1}^\infty\left(\fint_{\mathbb{T}\times \mathbb{R}}b_j\left(\widetilde{e_j}^\varepsilon(t,x)\otimes \widetilde{u}^\varepsilon(t,x+y)+\widetilde{u}^\varepsilon(t,x)\otimes \widetilde{e_j}^\varepsilon(t,x+y)\right)dx\right)d\beta_j.
\end{align*}
We denote by $\nabla^\top_y$ the transpose of $\nabla_y$, that is, $(\nabla^\top_y u)_{ij}=(\nabla_y u)_{ji}=\partial_{y_j}u^i$.
Integrating in the time interval $[0,T]$, taking expectation and pairing against the test function \eqref{testfunction} and integrating by parts gives 
\begin{align}\label{timedip}
&\int_{\mathbb{R}^2}\phi(y):\Gamma^\varepsilon(T,y)dy-\int_{\mathbb{R}^2}\phi(y):\Gamma^\varepsilon(0,y)dy\notag\\
=&-\frac{1}{2}\sum_{j=1}^2\int_0^T\int_{\mathbb{R}^2}\partial_j\phi(y):{D^j}^\varepsilon(t,y)dydt+2\nu\int_0^T\int_{\mathbb{R}^2}\Delta\phi(y):\Gamma^\varepsilon(t,y)dydt\notag\\
&-2\alpha\int_0^T\int_{\mathbb{R}^2}\phi(y):\Gamma^\varepsilon(t,y)dydt-2\int_0^T\int_{\mathbb{R}^2}\phi(y):\varTheta^\varepsilon(t,y)dydt\notag\\
&+\int_0^T\int_{\mathbb{R}^2}\phi(y):P^\varepsilon(t,y)dydt+2T\int_{\mathbb{R}^2}\phi(y):a^\varepsilon(y)dy
\end{align}
where we have defined the regularised quantities,
\begin{align*}
       \Gamma^\varepsilon(t,y)&:=\mathbb{E}\fint_{\mathbb{T}\times \mathbb{R}}\widetilde{u}^\varepsilon(t,x)\otimes \widetilde{u}^\varepsilon(t,x+y)dx,\\ \varTheta^\varepsilon(t,y)&:=\frac{1}{2}\mathbb{E}\fint_{\mathbb{T}\times \mathbb{R}}\widetilde{u}^\varepsilon(t,x)\otimes (f\widetilde{u^\perp})^\varepsilon(t,x+y)+(f\widetilde{u^\perp})^\varepsilon(t,x)\otimes \widetilde{u}^\varepsilon(t,x+y)dx,\\
    {D^i}^\varepsilon(t,y)&:=2\mathbb{E}\fint_{\mathbb{T}\times \mathbb{R}}\widetilde{(u^i u)}^\varepsilon(t,x)\otimes \widetilde{u}^\varepsilon(t,x+y)-\widetilde{u}^\varepsilon(t,x)\otimes\widetilde{(u^iu)}^\varepsilon(t,x+y)dx,\\
    P^\varepsilon(t,y)&:=\mathbb{E}\left(\nabla_y\fint_{\mathbb{T}\cross \mathbb{R}} \widetilde{p}^\varepsilon(t,x)\widetilde{u}^\varepsilon(t,x+y)dx-\nabla^\top_y\fint_{\mathbb{T}\cross \mathbb{R}} \widetilde{u}^\varepsilon(t,x)\widetilde{p}^\varepsilon(t,x+y)dx\right)\\
    a^\varepsilon(y)&:=\frac{1}{2}\fint_{\mathbb{T}\times \mathbb{R}}C^\varepsilon(x,x+y)dx.
\end{align*}
By stationarity we can write the above equation as
\begin{align}
&0=-\frac{1}{2}\sum_{j=1}^2\int_{\mathbb{R}^2}\partial_j\phi(y):{D^j}^\varepsilon(t,y)dy+2\nu\int_{\mathbb{R}^2}\Delta\phi(y):\Gamma^\varepsilon(y)dy\notag\\
&-2\alpha\int_{\mathbb{R}^2}\phi(y):\Gamma^\varepsilon(y)dy-2\int_{\mathbb{R}^2}\phi(y):\varTheta^\varepsilon(y)dy\notag\\
&+\int_{\mathbb{R}^2}\phi(y):P^\varepsilon(y)dy+2\int_{\mathbb{R}^2}\phi(y):a^\varepsilon(y)dy
\end{align}

Finally, we want to pass to the limit as $\varepsilon\to0$ in \eqref{timedip} to obtain \eqref{KHM}. 
By definition of the noise along with bounded convergence theorem, we have that $C^\varepsilon(x,x+y)\to C(x,x+y)$ converge locally uniformly in both $x$ and $y$ and therefore $a^\varepsilon(y)\to a(y)$. 

 We show that the contribution from the pressure $P^\varepsilon$ vanishes. Using the identity
\begin{align*}
    \nabla \cdot\phi(y)&=\nabla \cdot(\Phi(|y|)Id+\Psi(|y|)\hat{y}\otimes\hat{y}=\eta(|y|)\hat{y}
\end{align*}
where
\begin{align*}
    \eta(l)&=\Phi'(l)+\Psi'(l)+2l^{-1}\Psi(l),
\end{align*}
we have
\begin{align*}
    \int_{\mathbb{R}^2}\phi(y)&:P^\varepsilon(t,y)dy\\
    =&-\mathbb{E}\int\int \nabla \cdot \phi(y) \cdot \widetilde{p}^\varepsilon(t,x)(\widetilde{u}^\varepsilon(t,x+y)-\widetilde{u}^\varepsilon(t,x-y))dxdy\\
    =&-\mathbb{E}\int\int  \eta(|y|) \widetilde{p}^\varepsilon(t,x)(\widetilde{u}^\varepsilon(t,x+y)-\widetilde{u}^\varepsilon(t,x-y))\cdot\hat{y}dxdy\\
    =&-\mathbb{E}\int \widetilde{p}^\varepsilon(t,x)\left(\int_{\mathbb{R}^+}\eta(l)\left[\int_{|y|=l}(\widetilde{u}^\varepsilon(t,x+y)-\widetilde{u}^\varepsilon(t,x-y))\cdot\hat{y}dS(y)\right]dl\right)dx.
\end{align*}
By the divergence theorem, we conclude that 
\begin{equation*}
    \int_{|y|=l}(\widetilde{u}^\varepsilon(t,x+y)-\widetilde{u}^\varepsilon(t,x-y))\cdot\hat{y}dS(y)=2\int_{|y|\leq l}\nabla \cdot \widetilde{u}^\varepsilon(x+y)dy=0 
\end{equation*}
since $\widetilde{u}^\varepsilon$ is a divergence-free vector field.

Due to the properties of mollifiers we have almost everywhere in $\Omega\cross[0,T]$,
\begin{equation}\label{mollif property}
    \widetilde{u}^\varepsilon\to \widetilde{u} \text{ in }H(\mathbb{T}\cross \mathbb{R}).
\end{equation}
Then, for almost every $(\omega,t,y)\in \Omega\cross [0,T]\cross \mathbb{R}^2$
\begin{align}
    &\fint_{\mathbb{T}\cross\mathbb{R}} |\widetilde{u}^\varepsilon(x)\otimes\widetilde{u}^\varepsilon(x+y)-\widetilde{u}(x)\otimes\widetilde{u}(x+y)|dx\notag\\
    &=\fint_{\mathbb{T}\cross\mathbb{R}} |\widetilde{u}^\varepsilon(x)\otimes\widetilde{u}^\varepsilon(x+y)-\widetilde{u}^\varepsilon(x)\otimes\widetilde{u}(x+y)+\widetilde{u}^\varepsilon(x)\otimes\widetilde{u}(x+y)-\widetilde{u}(x)\otimes\widetilde{u}(x+y)|dx\notag\\
    &\leq 2\norm{\widetilde{u}}_{L^2_{\mathbb{T}\cross\mathbb{R}}}\norm{\widetilde{u}^\varepsilon-\widetilde{u}}_{L^2_{\mathbb{T}\cross\mathbb{R}}}
\end{align}
which vanishes as $\varepsilon\to 0$. This implies that 
\begin{equation}
    \widetilde{u}^\varepsilon(x)\otimes\widetilde{u}^\varepsilon(x+y)\to\widetilde{u}(x)\otimes\widetilde{u}(x+y)\ \ \ \text{in } L^1(\mathbb{T}\cross\mathbb{R}).
\end{equation}
Moreover, 
\begin{equation*}
\fint_{\mathbb{T}\cross\mathbb{R}}\widetilde{u}^\varepsilon(x)\otimes\widetilde{u}^\varepsilon(x+y)dx\leq \norm{u}^2_{L^2_{\mathbb{T}\cross I}}.
\end{equation*}
By dominated convergence theorem we conclude that for each $y\in\mathbb{R}^2$,
\begin{equation}
    \Gamma^\varepsilon(y)\to\Gamma(y).
\end{equation}
Finally, the uniform bound 
\begin{equation}\label{Gamma unif bounf}
    \left|\phi(y):\Gamma^\varepsilon(y)\right|\leq \mathbb{E}\norm{u}_{L^2_{\mathbb{T}\cross I}}^2<\infty
\end{equation}
and the dominated convergence theorem imply 
\begin{equation}
    \phi(y):\Gamma^\varepsilon(y)\to\phi(y):\Gamma(y) \ \ \ \text{in }L^1(\mathbb{R}^2). 
\end{equation}
Since we have these convergence properties, we can pass to the limit in all terms of \eqref{timedip} that involve $\Gamma^\varepsilon$. 

We proceed similarly for the other terms. 
We have 
\begin{align}
    &\fint_{\mathbb{T}\times \mathbb{R}}\left|(f\widetilde{u^\perp})^\varepsilon(t,x)\otimes \widetilde{u}^\varepsilon(t,x+y)-(f\widetilde{u^\perp})(t,x)\otimes \widetilde{u}(t,x+y)\right|dx\notag\\
    \leq&\fint_{\mathbb{T}\times \mathbb{R}} \left|((f\widetilde{u^\perp})^\varepsilon-(f\widetilde{u^\perp}))(t,x)\right|\otimes \left|\widetilde{u}^\varepsilon(t,x+y)\right|\notag\\
    &+\left|f\widetilde{(u^\perp)}(t,x)\right|\otimes \left|(\widetilde{u}^\varepsilon-\widetilde{u})(t,x+y)\right|dx\notag\\
    \leq&\norm{(f\widetilde{u^\perp})^\varepsilon-(f\widetilde{u^\perp})}_{L^2_{\mathbb{T}\cross\mathbb{R}}}\norm{\widetilde{u}^\varepsilon}_{L^2_{\mathbb{T}\cross\mathbb{R}}}+\norm{f\widetilde{(u^\perp)}}_{L^2_{\mathbb{T}\cross\mathbb{R}}}\norm{\widetilde{u}^\varepsilon-\widetilde{u}}_{L^2_{\mathbb{T}\cross\mathbb{R}}}\notag\\    
    \leq&c\left(\norm{(\widetilde{u^\perp})^\varepsilon-u^\perp}_{L^2_{\mathbb{T}\cross \mathbb{R}}}\norm{\widetilde{u}}_{L^2_{\mathbb{T}\cross \mathbb{R}}}+\norm{\widetilde{u^\perp}}_{L^2_{\mathbb{T}\cross \mathbb{R}}}\norm{\widetilde{u}^\varepsilon-\widetilde{u}}_{L^2_{\mathbb{T}\cross \mathbb{R}}}\right)
\end{align}
which vanishes since the norm is rotation invariant, \eqref{mollif property} and solution regularity.

This prove that almost everywhere in $(\omega,t,y)\in\Omega\cross[0,T]\cross \mathbb{R}^2$
\begin{equation}
    (f\widetilde{u^\perp})^\varepsilon(x)\otimes \widetilde{u}^\varepsilon(x+y)\to(f\widetilde{u^\perp})(x)\otimes \widetilde{u}(x+y)
\end{equation}
and 
\begin{equation}
    \widetilde{u}^\varepsilon(x)\otimes(f\widetilde{u^\perp})^\varepsilon (x+y)\to\widetilde{u}(x)\otimes(f\widetilde{u^\perp}) (x+y)
\end{equation}
converge in $L^1(\mathbb{T}\cross\mathbb{R})$.
Moreover, we have
\begin{equation}
    \fint_{\mathbb{T}\cross I} \widetilde{u}^\varepsilon(x)\otimes(f\widetilde{u^\perp})^\varepsilon (x+y) dx\lesssim \norm{u}^2_{L^2}.
\end{equation}
Then, the dominated convergence theorem implies for each $y\in \mathbb{R}^2$
\begin{equation}
    \varTheta^\varepsilon(y)\to\varTheta(y).
\end{equation}
Similar to \eqref{Gamma unif bounf} we have
\begin{equation}
    \left|\phi(y):\varTheta^\varepsilon(y)\right|\lesssim\mathbb{E}\norm{u}_{L^2_{\mathbb{T}\cross I}}^2<\infty.
\end{equation}
and by using the bounded convergence theorem, 
\begin{equation}\label{bounded conv thm}
    \phi(y):\varTheta^\varepsilon(y)\to\phi(y):\varTheta(y) \ \ \ \text{in }L^1( \mathbb{R}^2). 
\end{equation}
Then, we can pass to the limit in \eqref{timedip} for $\varTheta^\varepsilon$.

It remains to pass to the limit in the term involving ${D^j}^\varepsilon$.
As in the previous steps for almost every $(\omega,t,y)\in \Omega\cross[0,T]\cross \mathbb{R}^2$
\begin{align}\label{nonlinear term}
    &\fint_{\mathbb{T\cross I}}\left|\widetilde{(u^i u)}^\varepsilon(x)\otimes\widetilde{u}^\varepsilon(x+y)-\widetilde{u^i}\widetilde{u}(x)\otimes\widetilde{u}(x+y)\right|dx\notag\\
    \leq &\norm{\widetilde{u}}_{L^3_{\mathbb{T}\cross \mathbb{R}}}\norm{\widetilde{(u^iu)}^\varepsilon-\widetilde{u^i}\widetilde{u}}_{L^{3/2}_{\mathbb{T}\cross \mathbb{R}}}+\norm{\widetilde{u^i}}_{L^3_{\mathbb{T}\cross \mathbb{R}}}\norm{\widetilde{u}}_{L^3_{\mathbb{T}\cross \mathbb{R}}}\norm{\widetilde{u}-\widetilde{u}^\varepsilon}_{L^3_{\mathbb{T}\cross \mathbb{R}}}
\end{align}
Observing that $\widetilde{u}\in L^3(\mathbb{T}\cross \mathbb{R})$ (by Sobolev embedding theorem) due to the properties of mollifiers we have
\begin{equation}\label{l3 conv}
    \widetilde{u}^\varepsilon\to \widetilde{u} \ \ \ \text{ in }L^3(\mathbb{T\cross R})
\end{equation}
and since 
\begin{equation}
    \norm{\widetilde{u^iu}}_{L^{3/2}(\mathbb{T}\cross \mathbb{R})}\leq \norm{\widetilde{u}}^2_{L^{3}(\mathbb{T}\cross \mathbb{R})}< \infty
\end{equation}
we have
\begin{equation}\label{l3/2 conv}
    \norm{(\widetilde{u^iu})^\varepsilon-\widetilde{u^i}\widetilde{u}}_{L^{3/2}_{\mathbb{T}\cross \mathbb{R}}}\to 0.
\end{equation}
Then, by \eqref{nonlinear term}-\eqref{l3/2 conv} we derive 
\begin{equation*}
    \widetilde{(u^ju)^\varepsilon}(\omega,t,\cdot)\otimes \widetilde{u}^\varepsilon(\omega,t,\cdot+y)\to \widetilde{u}(\omega,t,\cdot)\otimes \widetilde{u}(\omega,t,\cdot+y) \widetilde{u^j}(\omega,t,\cdot)
\end{equation*}
and 
\begin{equation}\label{lim1}
    \widetilde{u}^\varepsilon(\omega,t,\cdot)\otimes \partial_{x^j}(\widetilde{u}^j \widetilde{u})^\varepsilon(\omega,t,\cdot+y)\to \widetilde{u}(\omega,t,\cdot)\otimes \widetilde{u}(\omega,t,\cdot+y) \widetilde{u^j}(\omega,t,\cdot+\varepsilon)
\end{equation}
in $L^1(\mathbb{T}\cross \mathbb{R})$.
Moreover, 
\begin{equation}
    \fint_{\mathbb{T\cross I}}\left|\widetilde{(u^i u)}^\varepsilon(x)\otimes\widetilde{u}^\varepsilon(x+y)\right|dx\lesssim \norm{\widetilde{u}}^3_{L^3_{\mathbb{T}\cross \mathbb{R}}}<\infty
\end{equation}
Then, the dominated convergence theorem tells us
\begin{equation}
    \mathbb{E}\fint_{\mathbb{T}\cross I}\widetilde{(u^i u)}^\varepsilon(x)\otimes\widetilde{u}^\varepsilon(x+y)dx\to \mathbb{E}\fint_{\mathbb{T}\cross I}\widetilde{u^i}\widetilde{u}(x)\otimes\widetilde{u}(x+y)dx
\end{equation}
and similarly for $\widetilde{u}^\varepsilon\otimes \partial_{x^j}(\widetilde{u}^j \widetilde{u})^\varepsilon$.

Due to Lemma 3.4 in \cite{bed3D} (which is still valid for how we defined the extension ) and using a change of variable, we can write the limit in terms of the third structure function 
\begin{align*}
   2\int_{\mathbb{R}^2}\partial_{h^j}\phi(y):\mathbb{E}\int_{\mathbb{T}\cross I}&\big(\widetilde{u}(\cdot,x)\otimes \widetilde{u}(\cdot, x+y)\big)\widetilde{u}^j(\cdot,x)\\
   -&\big(\widetilde{u}(\cdot,x)\otimes \widetilde{u}(\cdot,x+y)\big)\widetilde{u}^j(\cdot,x+y) dx dy=\int_{\mathbb{R}^2}\partial_{h^j}\phi(y):D^j(t,y)dy
\end{align*}
Finally, using the uniform bound 
\begin{equation*}
    \left|\partial_{h^j}\phi(y):{D^j}^\varepsilon(y)\right|\leq\left(\mathbb{E}\norm{u\otimes u}^2_{L^2}\right)^{1/2}\left(\mathbb{E}\norm{ u}^2_{L^2}\right)^{1/2}<\infty,
\end{equation*}
it follows from the dominated convergence theorem that
\begin{equation*}
    \partial_{h^j}\phi:{D^j}^\varepsilon\to\partial_{h^j}\phi:{D^j} \ \text{ in } L^1(\mathbb{R}^2)
\end{equation*}
and therefore we may pass the limit of the term ${D^j}^\varepsilon$ in \eqref{timedip}. 
\end{proof}

\begin{notat}
    Throughout the paper, we replace $\widetilde{u}$ by $u$ and similarly for the other extended functions above whenever it does not make any confusion. 
\end{notat}

\begin{remark}
    Since $supp (u)\subset \mathbb{T}\cross I$, the following equality holds 
    \begin{equation}
        \fint_{\mathbb{T\cross R}} u(x)dx=\fint_{\mathbb{T}\cross I} u(x)dx
    \end{equation}
    and similar when we have $p$ and any product between these functions.
\end{remark}
Let us define the spherically averaged energy flux structure-function,
\begin{equation}
    \overline{D}(l):=\mathbb{E}\fint_{\mathbb{S}}\fint_{\mathbb{T}\times I}|\delta_{ln}u(x)|^2\delta_{ln}u(x)\cdot n\ dxdn,
\end{equation}
and spherically averaged correlation functions,
\begin{align}
    \overline{\Gamma}(l)&:=\fint_{\mathbb{S}}\trace\Gamma(ln)dn,\\
    \overline{\varTheta}(l)&:=\fint_{\mathbb{S}}\trace\varTheta(ln)dn,\label{coriolis spherical av}\\
    \overline{a}(l)&:=\fint_{\mathbb{S}}\trace a(ln)dn.
\end{align}
Then we can write \eqref{KHM} as follows.

\begin{lem}\label{lemma:velocityKHM}
The following identity holds for each $l>0$
\begin{align}\label{relation}
    \overline{D}(l)=&-2\nu l\fint_{|y|\leq l}\Delta \trace\Gamma(y)dy+2\alpha l\fint_{|y|\leq l} \trace\Gamma(y)dy+2 l\fint_{|y|\leq l} \trace\varTheta(y)dy\notag\\
    &-2 l\fint_{|y|\leq l} \trace a(y)dy\notag\\
    =&-4\nu\overline{\Gamma}'(l)+\frac{4\alpha}{l}\int_0^l r\overline{\Gamma}(r)dr+\frac{4}{l}\int_0^l r\overline{\varTheta}(r)dr-\frac{4}{l}\int^l_0r\overline{a}(r)dr.
\end{align}
\end{lem}

\begin{proof}
    Let us assume that 
    \begin{equation}
    \phi(y)=\Phi(|y|)Id, 
\end{equation}
i.e. $\Psi(|y|)=0$ in \eqref{testfunction}. Then, we rewrite \eqref{KHM} as
\begin{align}\label{KHM solo phi}
    \sum^2_{j=1}\int_{\mathbb{R}^2}\partial_j\phi(y):D^j(y)dy&=4\nu\int_{\mathbb{R}^2}\Delta\phi(y):\Gamma(y)dy-4\alpha\int_{\mathbb{R}^2}\phi(y):\Gamma(y)dy\notag\\
    &-4\int_{\mathbb{R}^2}\phi(y):\varTheta(y)dy+4\int_{\mathbb{R}^2}\phi(y):a(y)dy
\end{align}
In spherical coordinates, \eqref{KHM solo phi} becomes 

\begin{align}\label{KHM sferiche}
    \int_{\mathbb{R}^+}\fint_\mathbb{S}l\partial_j\phi(l):D^j(ln)dldn&=4\nu\int_{\mathbb{R}^+}\fint_\mathbb{S}l\Delta\phi(l):\Gamma(ln)dldn\\
    &-4\alpha\int_{\mathbb{R}^+}\fint_\mathbb{S}l\phi(l):\Gamma(ln)dldn\notag\\
    &-4\int_{\mathbb{R}^+}\fint_\mathbb{S}l\phi(l):\varTheta(ln)dldn\notag\\
    &+4\int_{\mathbb{R}^+}\fint_\mathbb{S}l\phi(l):a(rn)dldn.\notag
\end{align}
Integrating by parts, we get
\begin{align}
    -\int_{\mathbb{R}^+}\left(\overline{D}'(l)+\frac{1}{l}\overline{D}(l)\right)l\Phi(l)dl&=4\nu\int_{\mathbb{R}^+}l\Phi(l)\left(\overline{\Gamma}''(l)+\frac{1}{l}\overline{\Gamma}'(l)\right)dl\\
    &-4\alpha\int_{\mathbb{R}^+}l\Phi(l)\overline{\Gamma}(l)dl-4\int_{\mathbb{R}^+}l\Phi(l)\overline{\Theta}(l)dl\notag\\ &+4\int_{\mathbb{R}^+}l\Phi(l)\overline{a}(l)dl.\notag
\end{align}
Then, the following ODE holds in the sense of distribution
\begin{equation*}
    -\partial_l\left(l^2\frac{\overline{D}(l)}{l}\right)=-4l\left[\nu\left(\overline{\Gamma}''(l)+\frac{1}{l}\overline{\Gamma}'(l)\right)-\alpha\overline{\Gamma}(l)-\overline{\varTheta}(l)+\overline{a}(l)\right].
\end{equation*}
Integrating in $l$, we have 
\begin{equation*}
    \frac{\overline{D}(l)}{l}=\frac{4}{l^2}\int_0^l \nu\left(r\overline{\Gamma}'(r)\right)'-r\alpha\overline{\Gamma}(r)-r\overline{\varTheta}(r)+r\overline{a}(r)dr
\end{equation*}
from which \eqref{relation}.
\end{proof}

\subsection{Modified Vorticity KHM relations}
Similarly to the previous section we derive the Modified Vorticity KHM relations in terms of spherically averaged coordinates. The new term is denoted by $\overline{\mathfrak{Q}}$.
We remark that only the direct enstrophy cascade can be written in the mixed velocity-vorticity structure function where the spherically averaged KHM fully describes the problem at small scale.

We define the two point correlation functions for the vorticity, curl of the Coriolis force and curl of the noise as
\begin{align}
    \mathfrak{C}(y)&=\mathbb{E}\fint_{\mathbb{T}\cross I} \omega(x)\omega(x+y) dx;\\
    \mathfrak{Q}(y)&=\frac{1}{2}\beta\mathbb{E}\fint_{\mathbb{T}\cross I}u^2(x)\omega(x+y)+u^2(x+y)\omega(x)dx;\\
    \mathfrak{a}(y)&=\frac{1}{2}\sum_j b_j^2\fint_{\mathbb{T}\cross I} \left(\nabla \cross{e_j(x)}\right)\left(\nabla \cross{e_j}(x+y)\right) dx;
\end{align}
as well as the corresponding enstrophy flux structure function,
\begin{equation}
    \mathfrak{D}(y)=\mathbb{E}\fint_{\mathbb{T}\cross I} \left|\delta_y\omega(x)\right|^2\delta_yu(x)dx.
\end{equation}
\begin{remark}\label{rmkkhmreg2}
As a consequence of Remark \ref{rmk1}, recalling that $\omega=\partial_1 u^2-\partial_2 u^1$, we can see that these quantities are at least $C^2$. 
\end{remark}

\begin{proposition} (Vorticity KHM relation)
Let $\omega$ be a statistically stationary solution to $\eqref{vorticityequation}$ as in \cite[Corollary 2.2]{GigliolaAmiraliYuriWP}. Then the following relation holds.
\begin{equation}\label{vorticitykhm}
    \nabla \cdot\mathfrak{D}(y)=-4\nu\Delta\mathfrak{C}(y)+4\alpha\mathfrak{C}(y)+4\beta\mathfrak{Q}(y)-4\mathfrak{a}(y).
\end{equation}
\end{proposition}

\begin{proof}
The proof is essentially the same as in Proposition \ref{weakvelocitykhmprop}, where we now have the correlation functions $\mathfrak{C},\mathfrak{Q},\mathfrak{a}$ and $\mathfrak{D}$. Once again, this relation is defined for the extended version of the functions as in the relation \eqref{KHM}. 
For analogy, we avoid using the heavy notation by considering only the regularization. 

Let us consider the mollifier stochastic equation 
\begin{align}\label{stochasticeqvorticity}
    &d\omega^\varepsilon(t,x)+\nabla \cdot (u \otimes \omega)^\varepsilon(t,x)dt-\nu \Delta \omega^\varepsilon(t,x) dt+\alpha \omega^\varepsilon(t,x) dt+ \beta{u^2}^\varepsilon (t,x)dt\notag\\
    =&d\nabla \cross\zeta^\varepsilon(t,x).
\end{align}

Let $y\in\mathbb{T}\times I$, as discussed in Proposition \ref{weakvelocitykhmprop} the evolution equation of $\omega^\varepsilon(t,x)\cdot \omega^\varepsilon(t,x+y)$ satisfies the stochastic product rule 
\begin{align}\label{productrulevorticity}
    d\left(\omega^\varepsilon(t,x)\cdot \omega^\varepsilon(t,x+y)\right)&=d\omega^\varepsilon(t,x)\cdot  \omega^\varepsilon(t,x+y)+\omega^\varepsilon(t,x)\cdot  d\omega^\varepsilon(t,x+y)\notag\\
    &+\left[\omega^\varepsilon(\cdot,x), \omega^\varepsilon(\cdot,x+y)\right](t),
\end{align}
where $\left[\omega^\varepsilon(\cdot,x), \omega^\varepsilon(\cdot,x+y)\right](t)$ is the cross variation of $\omega^\varepsilon(\cdot,x)$ and $\omega^\varepsilon(\cdot,x+y)$ and is given by
\begin{align*}
    \left[\omega^\varepsilon(\cdot,x), \omega^\varepsilon(\cdot,x+y)\right](t)&=\left[\nabla \cross\zeta^\varepsilon(\cdot,x),\nabla \cross\zeta^\varepsilon(\cdot,x+y)\right](t)\\
    &=\sum_{i,j}b_i b_j\left({\nabla \cross e_i}^\varepsilon \cdot {\nabla \cross e_j}^\varepsilon(x+y)\right)\left[\beta_i,\beta_j\right](t)\\
    &=tC^\varepsilon(x,x+y),
\end{align*}
with,
\begin{equation*}
    C^\varepsilon(x,x+y)=\sum_{j=1}^\infty b_j^2 {\nabla \cross e_j}^\varepsilon(x)\cdot  \nabla \cross e^\varepsilon_j(x+y).
\end{equation*}
Upon integrating \eqref{productrulevorticity} in $x$ and $t$, using \eqref{stochasticeqvorticity} and taking expectation we get 
\begin{align}\label{timedipvorticity}
\mathfrak{C}^\varepsilon(T,y)-\mathfrak{C}^\varepsilon(0,y)
=&-\frac{1}{2}\sum_{j=1}^2\int_0^T\partial_j{\mathfrak{D}^j}^\varepsilon(t,y)dt+2\nu\int_0^T\Delta\mathfrak{C}^\varepsilon(t,y)dt\notag\\&-2\alpha\int_0^T\mathfrak{C}^\varepsilon(t,y)dt-2\int_0^T\mathfrak{Q}^\varepsilon(t,y)dt+2T\mathfrak{a}^\varepsilon(y)
\end{align}
where we have defined the regularised quantities,
\begin{align*}
       \mathfrak{C}^\varepsilon(t,y)&:=\mathbb{E}\fint_{\mathbb{T}\times I}\omega^\varepsilon(t,x)\cdot \omega^\varepsilon(t,x+y)dx,\\ \mathfrak{Q}^\varepsilon(t,y)&:=\frac{1}{2}\beta\mathbb{E}\fint_{\mathbb{T}\times I}{u^2}^\varepsilon(t,x)\cdot \omega^\varepsilon(t,x+y)+\omega^\varepsilon(t,x)\cdot {u^2}^\varepsilon(t,x+y)dx,\\
    {\mathfrak{D}^i}^\varepsilon(t,y)&:=2\mathbb{E}\fint_{\mathbb{T}\times I}(u^i \omega)^\varepsilon(t,x)\cdot \omega^\varepsilon(t,x+y)-\omega^\varepsilon(t,x)\cdot(u^i\omega)^\varepsilon(t,x+y)dx,\\
    \mathfrak{a}^\varepsilon(y)&:=\frac{1}{2}\fint_{\mathbb{T}\times I}C^\varepsilon(x,x+y)dx.
\end{align*}
As seen in Proposition \ref{weakvelocitykhmprop} we can pass to the limit as $\varepsilon\to0$, thanks to the $C^2$ regularity mentioned in 
Remark \ref{rmkkhmreg2}. Moreover, if we assume stationarity, we get \eqref{vorticitykhm}.
\end{proof}

As well as for the velocity relations in the previous section, we define the spherically averaged enstrophy flux structure function 
\begin{equation}
    \overline{\mathfrak{D}}(l):=\mathbb{E}\fint_{\mathbb{S}}\fint_{\mathbb{T}\times I}|\delta_{ln}\omega(x)|^2\delta_{ln}u(x)\cdot n\ dxdn,
\end{equation}
and spherically averaged correlation functions,
\begin{align}
    \overline{\mathfrak{C}}(l)&:=\fint_{\mathbb{S}}\mathfrak{C}(ln)dn,\\
    \overline{\mathfrak{Q}}(l)&:=\fint_{\mathbb{S}}\mathfrak{Q}(ln)dn,\\
    \overline{\mathfrak{a}}(l)&:=\fint_{\mathbb{S}}\mathfrak{a}(ln)dn.
\end{align}

\begin{lem}\label{lemma:vorticityKHM}
The following identity holds for each $l>0$
\begin{align}\label{sphericallyaveragedvorticity}
    \overline{\mathfrak{D}}(l)=&-2\nu l\fint_{|y|\leq l}\Delta \mathfrak{C}(y)dy+2\alpha l\fint_{|y|\leq l} \mathfrak{C}(y)dy+2 l\fint_{|y|\leq l} \mathfrak{Q}(y)dy-2 l\fint_{|y|\leq l} \mathfrak{a}(y)dy\notag\\
    =&-4\nu\overline{\mathfrak{C}}'(l)+\frac{4\alpha}{l}\int_0^l r\overline{\mathfrak{C}}(r)dr+\frac{4}{l}\int_0^l r\overline{\mathfrak{Q}}(r)dr-\frac{4}{l}\int^l_0r\overline{\mathfrak{a}}(r)dr
\end{align}
\end{lem}
\begin{proof}
As observed in the proof of \eqref{relation}, using the divergence theorem and integrating both sides of \eqref{vorticitykhm} over $\{|y|\leq l\}$, we obtain the formula for the spherically averaged structure function \eqref{sphericallyaveragedvorticity}.
\end{proof}

\section{Direct cascade}\label{sectiondirect}
In this section, we prove Theorem \ref{thm:enstrophycascade} starting from the KHM relations in the form presented in Lemmas \ref{lemma:velocityKHM} and \ref{lemma:vorticityKHM}. The line of proof mainly follows what we see in \cite{bed2D} with a noticeable difference due to the inclusion of the Coriolis force and thus the respective $\overline{\varTheta}$ and $\mathfrak{Q}$ terms in \eqref{relation} and \eqref{sphericallyaveragedvorticity}. By using the regularity results concerning the invariant measure's support we prove that small scale contribution of these therms is zero on the third order structure functions. 
\begin{proof}[Proof of \eqref{directcascade3.1a}]
We write the equation for the spherically averaged flux $\overline{\mathfrak{D}}(l)$ as
\begin{equation}
    \frac{\overline{\mathfrak{D}}(l)}{l}=-\frac{4\nu\overline{\mathfrak{C}}'(l)}{l}+\frac{4\alpha}{l^2}\int_0^l r\overline{\mathfrak{C}}(r)dr+\frac{4}{l^2}\int_0^l r\overline{\mathfrak{Q}}(r)dr-\frac{4}{l^2}\int^l_0r\overline{\mathfrak{a}}(r)dr
\end{equation}
\textbf{Step 1.} Firstly, we show that 
\begin{equation}
    \lim_{\nu\to 0}\sup_{\alpha\in(0,1)}\sup_{l\in(l_\nu,1)}\left|\frac{4\nu\overline{\mathfrak{C}}'(l)}{l}\right|=0,
\end{equation}
for all $l_\nu$ satisfying \eqref{conditionlv}. In fact, by definition of $\overline{\mathfrak{C}}(l)$ we have, 
\begin{equation}
    \overline{\mathfrak{C}}'(l)=\sum_{i=1}^2\mathbb{E}\fint_{\mathbb{S}}\fint_{\mathbb{T}\cross I} n^i\omega(x)\delta_{x^i}\omega(x+ln)dxdn,
\end{equation}
where the prime denotes the $l$-derivative. Using Cauchy-Schwarz inequality we get,
\begin{equation}
    \left|\overline{\mathfrak{C}}'(l)\right|\leq\mathbb{E}\norm{\nabla \omega}_{L^2}\norm{\omega}_{L^2}.
\end{equation}
Then, using the enstrophy balance \eqref{enstrophy balance},
\begin{equation}
    \left|\frac{4\nu\overline{\mathfrak{C}}'(l)}{l}\right|\leq\frac{1}{l_\nu}\left(\nu\mathbb{E}\norm{\nabla \omega}^2_{L^2}\right)^{\frac{1}{2}}\left(\nu\mathbb{E}\norm{ \omega}^2_{L^2}\right)^{\frac{1}{2}}\leq\frac{\eta^{1/2}}{l_\nu}\left(\nu\mathbb{E}\norm{\omega}^2_{L^2}\right)^{\frac{1}{2}}
\end{equation}
that vanishes as $\nu\to 0$ by \eqref{uniform bound on vorticity} and \eqref{conditionlv}.

\textbf{Step 2.} We prove vanishing of the damping over the inertial range, i.e.,
\begin{equation}
    \lim_{\alpha\to 0} \sup_{\nu\in(0,1)}\left|\frac{4\alpha}{l^2}\int^l_0r\overline{\mathfrak{C}}(r)dr\right|=0.
\end{equation}
By definition and using Cauchy-Schwarz inequality, we have 
\begin{equation}
    \left|\overline{\mathfrak{C}}(l)\right|\leq\mathbb{E}\norm{\omega}^2_{L^2}.
\end{equation}
Then,
\begin{equation}
    \left|\frac{4\alpha}{l^2}\int^l_0r\overline{\mathfrak{C}}(r)dr\right|\leq\frac{4\alpha}{l^2}\mathbb{E}\norm{\omega}^2_{L^2}\frac{l^2}{2}=2\alpha\mathbb{E}\norm{\omega}^2_{L^2}.
\end{equation}
which vanishes by \eqref{uniform bound on vorticity}.

\textbf{Step 3.} We now turn to the last term and we prove that,
\begin{equation}\label{mathfraka}
    \lim_{l_I\to 0}\sup_{l\in(0,l_I)}\left|\frac{4}{l^2}\int_0^l r\overline{\mathfrak{a}}(r)dr-2\eta\right|=0.
\end{equation}
By the regularity of the noise, we can Taylor expand $\overline{\mathfrak{a}}$, 
\begin{equation}
    \overline{\mathfrak{a}}(l)=\overline{\mathfrak{a}}(0)+O(l).
\end{equation}
Moreover, by definition, 
\begin{equation}
    \overline{\mathfrak{a}}(0)=\frac{1}{2}\sum_{j=1}^2b_j^2 \fint_{\mathbb{S}}\fint_{\mathbb{T}\cross I}\left|\nabla \cross e_j\right|^2 dxdn=\eta
\end{equation}
Then,
\begin{align}
    \left|\frac{4}{l^2}\int_0^l r\overline{\mathfrak{a}}(r)dr\right|&=\left|\frac{4}{l^2}\int_0^l r\left(\overline{\mathfrak{a}}(r)-\overline{\mathfrak{a}}(0)+\overline{\mathfrak{a}}(0)\right)dr\right|\notag\\
    &\leq\frac{4}{l^2}\int_0^l r \left|\overline{\mathfrak{a}}(r)-\overline{\mathfrak{a}}(0)\right|dr+\frac{4}{l^2}\int_0^l r\left|\overline{\mathfrak{a}}(0)\right|dr\notag\\
    &\leq \frac{4}{l^2}\int_0^l O(r^2)dr+\frac{4}{l^2}\frac{l^2}{2}\eta\notag\\
    &=O(l)+2\eta
\end{align}
and \eqref{mathfraka} follows immediately.

\textbf{Step 4.} Finally, we prove 
\begin{equation}
    \lim_{l_I\to0}\limsup_{\nu,\alpha\to 0}\sup_{l\in[l_\nu,l_I]}\left|\frac{4}{l^2}\int_0^l r\overline{\mathfrak{Q}}(r)dr\right|= 0.
\end{equation}
In fact, by definition we have 
\begin{align}
    \overline{\mathfrak{Q}}(0)&=\frac{\beta}{2}\mathbb{E}\fint_{\mathbb{S}}\fint_{\mathbb{T}\cross I}u^2(x)\omega(x)+u^2(x)\omega(x)dxdn\notag\\
    &=\beta\mathbb{E}\fint_{\mathbb{T}\cross I}u^2\partial_1 u^2dx-\beta\mathbb{E}\fint_{\mathbb{T}\cross I}u^2\partial_2 u^1dx\notag\\
    &=0+\beta\mathbb{E}\fint_{\mathbb{T}\cross I}u^1\partial_2 u^2dx\notag\\
    &=-\beta\mathbb{E}\fint_{\mathbb{T}\cross I}u^1\partial_1 u^1dx=0.
\end{align}
Moreover, 
\begin{align}
    \overline{\mathfrak{Q}}'(l)&=\frac{\beta}{2}\mathbb{E}\fint_{\mathbb{S}}\fint_{\mathbb{T}\cross I}\sum_{j=1}^2u^2(x)\partial_{x^j}\omega(x+ln)n^j+\omega(x)\partial_{x^j}u^2(x+ln)n^j.
\end{align}
After an integration by parts and using Cauchy-Schwartz inequality, we deduce
\begin{equation}
    \left|\overline{\mathfrak{Q}}'(l)\right|\lesssim\mathbb{E}\norm{\nabla u}_{L^2}\norm{\omega}_{L^2}=\mathbb{E}\norm{\omega}^2_{L^2}<\infty.
\end{equation}
We remark that the last estimate is $l$ independent. Since,
\begin{equation}
    \left|\overline{\mathfrak{Q}}(l)\right|\leq\left|\overline{\mathfrak{Q}}(0)\right|+\sup\left|\overline{\mathfrak{Q}}'(l)\right|l\leq Cl,
\end{equation}
we have
\begin{equation}
    \left|\frac{4}{l^2}\int_0^l r\overline{\mathfrak{Q}}(r)dr\right|\leq\frac{C}{l^2}\int_0^l r^2dr=Cl
\end{equation}
which vanishes as $l\to 0$.
\end{proof}

\begin{proof}[Proof of \eqref{directcascade3.1b}]
Let us consider the equality \eqref{relation} and divide it by $l^3$, we have
\begin{equation}
    \frac{\overline{D}(l)}{l^3}=-\frac{4\nu}{l^3}\overline{\Gamma}'(l)+\frac{4\alpha}{l^4}\int_0^l r\overline{\Gamma}(r)dr+\frac{4}{l^4}\int_0^l r\overline{\varTheta}(r)dr-\frac{4}{l^4}\int^l_0r\overline{a}(r)dr.
\end{equation}
\textbf{Step 1.} We first show that 
\begin{equation}\label{firstestimate}
    \frac{4}{l^4}\int_0^l r\overline{a}(r)dr=\frac{2\varepsilon}{l^2}-\frac{\eta}{4}+o_{l\to0}(1),
\end{equation}
where $\overline{a}$ is defined by
\begin{equation}
    \overline{a}(r)=\frac{1}{2}\sum_{j}b_j^2\fint_{\mathbb{S}}\fint_{\mathbb{T}\times I}e_j(x)\cdot e_j(x+rn)dx.
\end{equation}
We Taylor expand the factor $e_j(x+rn)$,
\begin{equation}\label{taylorexp}
    e_j(x+rn)=e_j(x)+rn\cdot\nabla e_j(x)+\frac{1}{2}rn\cdot He_j(x)rn+o(r^2)
\end{equation}
where $He_j$ is the Hessian matrix of $e_j$ and $o$ is the classic little $o$-notation, i.e., for two general functions $f(x),g(x)$, if $\lim_{x\to \infty}\frac{f(x)}{g(x)}=0$, then $f(x)=o(g(x))$.

We observe that 
\begin{equation}\label{propertyderiv}
    \int_{\mathbb{T}\cross I}e_j\cdot\partial_{x^i}e_jdx=0.
\end{equation}
Then,
\begin{align*}
    \fint_{\mathbb{T}\cross I} e_j(x)\cdot e_j(x+rn&)dx\\
    =&\fint_{\mathbb{T}\cross I} e_j(x)\left(e_j(x)+rn\cdot\nabla e_j(x)+\frac{1}{2}rn\cdot He_j(x)rn+o(r^2)\right)dx\\
    =&\fint_{\mathbb{T}\cross I} |e_j(x)|^2+\frac{r^2}{2}\sum_{i,m=1}^2 n^i n^m e_j(x)\cdot\partial_{x^m}\partial_{x^i}e_j(x)\ dx+o(r^3)\\
    =&\fint_{\mathbb{T}\cross I} |e_j(x)|^2-\frac{r^2}{2}\sum_{i,m=1}^2 n^i n^m \partial_{x^m}e_j(x)\cdot\partial_{x^i}e_j(x)\ dx+o(r^3).
\end{align*}
Since,
\begin{equation}\label{2nprop}
    \fint_{\mathbb{S}} n^in^m \ dn=\frac{1}{2}\delta_{im},
\end{equation}
we have
\begin{align*}
    \frac{1}{2}\sum_j b_j^2\fint_\mathbb{S}&\fint_{\mathbb{T}\cross I}e_j(x)\cdot e_j(x+rn)dxdn\\
    &=\frac{1}{2}\sum_j b_j^2 \fint_\mathbb{S}\fint_{\mathbb{T}\cross I}|e_j(x)|^2-\frac{r^2}{2}\sum_{i,m}n^i n^m\partial_{x^m}e_j(x)\cdot\partial_{x^i}e_j(x)\ dxdn+o(r^3)\\
    &=\frac{1}{2}\sum_j b_j^2 \fint_{\mathbb{T}\cross I}|e_j(x)|^2-\frac{r^2}{4}|\nabla e_j(x)|^2\ dx+o(r^3)
\end{align*}
where 
\begin{equation*}
    |\nabla e_j|^2=(\partial_{x^1}e_j^1)^2+(\partial_{x^2}e_j^2)^2+(\partial_{x^1}e_j^2)^2+(\partial_{x^2}e_j^1)^2.
\end{equation*}
By incompressibility we can rewrite the above equality as 
\begin{equation*}
    |\nabla e_j|^2=(\partial_{x^1}e_j^2)^2+(\partial_{x^2}e_j^1)^2-2\partial_{x^1}e_j^1\partial_{x^2}e_j^2.
\end{equation*}
Using integration by parts we get,
\begin{align}\label{propietacurl}
    \frac{1}{2}&\sum_j b_j^2\fint_\mathbb{S}\fint_{\mathbb{T}\cross I}e_j(x)\cdot e_j(x+rn)dxdn\\
    =\frac{1}{2}&\sum_j b_j^2 \left(\fint_{\mathbb{T}\cross I}|e_j(x)|^2dx-\frac{r^2}{4}\fint_{\mathbb{T}\cross I}|\nabla e_j(x)|^2\ dx\right)+o(r^3)\notag\\
    =\frac{1}{2}&\sum_j b_j^2 \left(\fint_{\mathbb{T}\cross I}|e_j(x)|^2dx-\frac{r^2}{4}\fint_{\mathbb{T}\cross I}(\partial_{x^1}e_j^2)^2+(\partial_{x^2}e_j^1)^2-2\partial_{x^2}e_j^1\partial_{x^1}e_j^2\ dx\right)+o(r^3)\notag\\
    =\frac{1}{2}&\sum_j b_j^2 \left(\fint_{\mathbb{T}\cross I}|e_j(x)|^2dx-\frac{r^2}{4}\fint_{\mathbb{T}\cross I}|\nabla \cross e_j(x)|^2\ dx\right)+o(r^3)\notag
\end{align}
Therefore, we have 
\begin{align*}
    \overline{a}(r)&=\frac{1}{2}\sum_jb_j^2\fint_{\mathbb{T}\cross I}|e_j(x)|^2 dx-\frac{r^2}{8}\sum_jb_j^2\fint_{\mathbb{T}\cross I}|\nabla \cross{e_j}|^2 dx+o(r^3)\\
    &=\varepsilon-\frac{r^2}{4}\eta+o(r^3)
\end{align*}
Multiplying by $r$ and integrating we obtain \eqref{firstestimate}.\\
\textbf{Step 2.} We want to show for $l_\nu$ satisfying $\eqref{conditionlv}$ 
\begin{equation}
    \lim_{l_I\to 0} \limsup_{\nu,\alpha\to 0} \sup_{l\in(l_\nu,l_I)}\left|\frac{4\alpha}{l^4}\int_0^l r\overline{\Gamma}(r)dr-\frac{2\varepsilon}{l^2}\right|=0.
\end{equation}
We remark that in the above formula, the order in which the $\nu$ and $\alpha$ are taken to $0$ does not matter.
We can write 
\begin{equation*}
    \frac{4\alpha}{l^4}\int_0^l r\overline{\Gamma}(r)dr=\frac{4\alpha}{l^4}\int_0^l r\left(\overline{\Gamma}(r)-\overline{\Gamma}(0)\right)dr+\frac{2\alpha}{l^2}\overline{\Gamma}(0)
\end{equation*}
Using the energy balance \eqref{energy balance} we have
\begin{equation*}
    \alpha\overline{\Gamma}(0)=\varepsilon-\nu\mathbb{E}\norm{\nabla u}^2.
\end{equation*}
Noting that  
\begin{equation}\label{propvorticty}
    \norm{\nabla u}^2_{L^2}=\norm{\omega}^2_{L^2},
\end{equation}
and recalling the conditions \eqref{uniform bound on vorticity} and \eqref{conditionlv}, we get
\begin{equation*}
    \alpha\overline{\Gamma}(0)=\varepsilon-\nu\mathbb{E}\norm{\omega}^2=\varepsilon+o(l_\nu^2).
\end{equation*}
Moreover, using a chain rule,
\begin{equation}\label{derivativecomp}
    \overline{\Gamma}'(l)=\sum_{i,j}\mathbb{E}\fint_{\mathbb{S}}\fint_{\mathbb{T}\cross I}n^i\partial_{x^i}u^j(x+ln)u^j(x)\ dxdn
\end{equation}
By \eqref{propertyderiv} we observe that $\overline{\Gamma}'(0)=0$. Furthermore, taking a second derivative of $\overline{\Gamma}$ and integrating by parts, we have 
\begin{equation*}
    \overline{\Gamma}''(l)=\sum_j\mathbb{E}\fint_{\mathbb{S}}\fint_{\mathbb{T}\cross I}\sum_{i,m} n^i n^m \partial_{x^i}u^j(x+ln)\partial_{x^m}u^j(x)\ dxdn.
\end{equation*}
By Cauchy–Schwarz inequality,
\begin{equation*}
   |\overline{\Gamma}''(l)|\leq C\mathbb{E}\norm{\nabla u}^2_{L^2}=C \mathbb{E}\norm{\omega}^2_{L^2}.
\end{equation*}
Using Taylor expansion for $\overline{\Gamma}(l)$,
\begin{equation*}
    \overline{\Gamma}(l)\approx \overline{\Gamma}(0)+\underbrace{ \overline{\Gamma}'(0)}_{=0}\cdot l+\frac{1}{2}\sup  \overline{\Gamma}''(l)\cdot l^2.
\end{equation*}
Then,
\begin{equation}
    \left| \overline{\Gamma}(l)- \overline{\Gamma}(0)\right|\lesssim \frac{1}{2}\sup  \overline{\Gamma}''(l)\cdot l^2\leq Cl^2\mathbb{E}\norm{\omega}^2_{L^2},
\end{equation}
and in particular 
\begin{equation}
    \frac{4\alpha}{l^4}\int_0^l r\left| \overline{\Gamma}(r)- \overline{\Gamma}(0)\right|dr\leq C\alpha\mathbb{E}\norm{\omega}^2_{L^2}\to 0,
\end{equation}
when $\alpha\to0$ due to \eqref{uniform bound on vorticity}.

\textbf{Step 3.} Analogous to \cite{bed3D} we have 
\begin{equation} \label{Step3eq}
    \limsup_{\nu,\alpha\to 0} \sup_{l\in(l_\nu,l_I)}\frac{4\nu}{l^3}\overline{\Gamma}'(l)=0.
\end{equation}
To show the above limit, first observe that: 
\begin{equation}
    \overline{\Gamma}'(l)-\overline{\Gamma}'(0)= \int_0^l \overline{\Gamma}''(l')d l'.
\end{equation}
However, in the previous step we already observed that $\overline{\Gamma}'(0)=0$, and\\ 
$\sup_{l' \in [0,l]} |\overline{\Gamma}''(l')| \leq C \mathbb{E}\norm{\omega}^2_{L^2}$. Therefore, we get: 
\begin{equation}
 |\overline{\Gamma}'(l)| \leq   l C    \mathbb{E}\norm{\omega}^2_{L^2}
\end{equation}
This means 
$$\frac{4 \nu}{l^3} \overline{\Gamma}'(l) \lesssim 
\frac{ \nu \mathbb{E}\norm{\omega}^2_{L^2}}{l^2}.  $$
This gives us \eqref{Step3eq} thanks to \eqref{uniform bound on vorticity} and \eqref{conditionlv}.

\textbf{Step 4.} We show that
\begin{equation}
    \lim_{l_I\to0}\limsup_{\nu,\alpha\to 0}\sup_{l\in[l_\nu,l_I]}\left|\frac{4}{l^4}\int_0^l r\overline{\varTheta}(r)dr\right|=0.
\end{equation}
By definition 
\begin{equation}\label{explicit theta}
    \overline{\varTheta}(l):=\mathbb{E}\fint_{\mathbb{S}}\fint_{\mathbb{T}\cross I} \beta ln^2\left(u^2(x)u^1(x+ln)-u^1(x)u^2(x+ln)\right) dxdn.
\end{equation}
Let us consider 
\begin{equation}
    \mu(l):=\mathbb{E}\fint_{\mathbb{S}}\fint_{\mathbb{T}\cross I}n^2 \left(u^2(x)u^1(x+ln)-u^1(x)u^2(x+ln)\right) dxdn.
\end{equation}
By definition,
\begin{equation}
    \mu(0)=0.
\end{equation}
Using a chain rule, we find
\begin{align*}
    \mu'(l)=\mathbb{E}\fint_{\mathbb{S}}\fint_{\mathbb{T}\cross I}&n^2u^2(x)\left(n^1\partial_{x^1}u^1(x+ln)+n^2\partial_{x^2}u^1(x+ln)\right)\\
    -&n^2u^1(x)\left(n^1\partial_{x^1}u^2(x+ln)+n^2\partial_{x^2}u^2(x+ln)\right)dxdn.
\end{align*}
Then,
\begin{align*}
    \mu'(0)=\mathbb{E}\fint_{\mathbb{S}}\fint_{\mathbb{T}\cross I}&n^2u^2(x)\left(n^1\partial_{x^1}u^1(x)+n^2\partial_{x^2}u^1(x)\right)\\-&n^2u^1(x)\left(n^1\partial_{x^1}u^2(x)+n^2\partial_{x^2}u^2(x)\right)dxdn.
\end{align*}
Using the incompressibility property, i.e. $\partial_{x^1}u^1(x)+\partial_{x^2}u^2(x)=0$, and integration by parts we have 
\begin{align*}
    \mu'(0)=\mathbb{E}\fint_{\mathbb{S}}\fint_{\mathbb{T}\cross I}&\underbrace{-n^1n^2u^2(x)\partial_{x^2}u^2(x)}_{=0}+(n^2)^2u^2\partial_{x^2}u^1(x)\\
    -&n^1n^2u^1(x)\partial_{x^1}u^2(x)+\underbrace{(n^2)^2u^1\partial_{x^1}u^1(x)}_{=0}dxdn.
\end{align*}
For the other two terms, after integration by parts, we conclude in the same way. In the end, we proved that 
\begin{equation}
    \mu'(0)=0.
\end{equation}
Now, we want to compute the second derivative of $\mu''$;
\begin{align*}
    |\mu''(l)|=\Bigg|\mathbb{E}\fint_{\mathbb{S}}\fint_{\mathbb{T}\cross I} &n^1n^2\partial_{x^1}u^1(x)\left(n^1\partial_{x^1}u^2(x+ln)+n^2\partial_{x^2}u^2(x+ln)\right)\\
    +&(n^2)^2\partial_{x^2}u^1(x)\left(n^1\partial_{x^1}u^2(x+ln)+n^2\partial_{x^2}u^2(x+ln)\right)\\
    -&n^1n^2\partial_{x^1}u^2(x)\left(n^1\partial_{x^1}u^1(x+ln)+n^2\partial_{x^2}u^1(x+ln)\right)\\
    -&(n^2)^2\partial_{x^2}u^2(x)\left(n^1\partial_{x^1}u^1(x+ln)+n^2\partial_{x^2}u^1(x+ln)\right)dxdn\Bigg|.
\end{align*}
Using Cauchy–Schwarz inequality and \eqref{uniform bound on vorticity}, we deduce,
\begin{equation*}
    |\mu''(l)|\leq C_1\mathbb{E}\norm{\nabla u}^2_{L^2}<\infty.
\end{equation*}
Since, 
\begin{equation}
    |\mu(l)|\leq|\mu(0)|+|\mu'(0)l|+\sup|\mu''(l)|l^2\leq C l^2,
\end{equation}
and thus 
\begin{equation}
    |\overline{\varTheta}(l)|\leq C|\beta l^3|,
\end{equation}
we have,
\begin{equation}
    \left|\frac{4}{l^4}\int_0^l r\overline{\varTheta}(r)dr\right|\leq \frac{4}{l^4}\int_0^l C|\beta r^4| dr\leq C_0\frac{1}{l^4}l^5
\end{equation}
that goes to 0 when $l\to 0$.
\end{proof}

\textbf{ Acknowledgment} Y.C. was funded by PRIN-MUR grant 2022YXWSLR "Boundary analysis for dispersive and viscous fluids". A.H. was funded in part by the FWO grant G098919N, the ANR grant LSD-15-CE40-0020-01, and the NSF Grant DMS-1929284. G.S. was funded in part by the NSF grant DMS-2052651 and the Simons Foundation through the Simons Collaboration  Wave Turbulence Grant.

\textbf{Conflict of interest.} The authors declare that they have no conflict of interest.

\textbf{Data availability.} Data sharing is not applicable.
We do not analyse or generate any datasets, because our work proceeds within a theoretical and mathematical approach.

\bibliographystyle{sn-basic}
\bibliography{sn-bibliography}

\end{document}